%% file: correlation-siam.tex
\documentclass[review,onefignum,onetabnum]{siamart171218}
\usepackage{amsmath,amssymb,a4wide,mathtools,bm}
\usepackage{enumitem}
\usepackage{csquotes}
\usepackage{algorithm,algpseudocode}

\usepackage{bm}
\usepackage[normalem]{ulem}
\usepackage{soul}

\newtheorem{assumption}[theorem]{Assumption}
\newtheorem{remark}{Remark}
\newtheorem{discussion}{Discussion}
\newtheorem{example}{Example}

\allowdisplaybreaks

\usepackage{tikz}
\usepackage{xparse}
\usetikzlibrary{tikzmark,decorations.pathreplacing,calc,decorations.markings}

\definecolor{colori}{RGB}{166,35,41}
\definecolor{colorii}{RGB}{41,35,166}

\NewDocumentCommand\MyArrow{O{0pt}mmmO{out=150,in=210}}
{%
\begin{tikzpicture}[overlay, remember picture]
  \draw [->,thick,line width=2pt,#4]
    ( $ ({pic cs:#3}|-{pic cs:#2}) + (-#1,1ex) $ ) to[#5]  
    ( $ (pic cs:#3) + (-#1,0) $ );
\end{tikzpicture}%
}

\input{definition}

\begin{document}

\input{title}

\maketitle

\input{content}

\section*{Acknowledgments}The author thanks Thorsten Hohage for inspiring discussions about passive imaging and Christian Aarset for valuable insights on covariance operators and for thorough proofreading. The author wishes to thank the reviewers for their insightful suggestions, leading to an improvement of the manuscript. The author acknowledges support from the DFG through Grant 432680300 - SFB 1456 (C04).

\input{appendix}

%\printbibliography
\bibliographystyle{siamplain}
\bibliography{lit.bib}
\end{document}

%% file: definition.tex
\def\({\left(}
\def\){\right)}
\def\l[{\left[}
\def\r]{\right]}

\newcommand{\compt}{\hookrightarrow\mathrel{\mspace{-15mu}}\rightarrow}
\newcommand{\embed}{\hookrightarrow}

\newcommand{\blue}[1]{\textcolor{black}{#1}}

\newcommand{\re}[1]{\Re\left\{{#1}\right\}}
\newcommand{\im}[1]{\Im\left\{{#1}\right\}}

\def\dom{\Omega}
\def\doms{\Omega\times\Omega}
\def\R{\mathbb{R}}
\def\C{\mathbb{C}}

\def\ubj{\overline{u_j}}

\def\Sv{\textbf{\textit S}}
\def\Bc{\mathcal{B}}	% manifold 
\def\Xhat{{\widehat{X}}}

\def\What{{\widehat{W}}}
\def\Uhat{{\widehat{U}}}
\def\Uhat{{\widehat{U}}}
\def\Mhat{\widehat{M}}
\def\Nhat{\widehat{N}}
\def\Lhat{\widehat{L}}
\def\thetil{\widetilde{\theta}}
\def\thedag{\theta^\dagger}
\def\pihat{\widehat{\pi}}

\def\E{\mathbb{E}}
\DeclareMathOperator{\cov}{Cov_J}

\def\sumj{\frac{1}{J}\sum_{j=1}^J}
\def\sumi{\sum_{i=1}^I}

\def\HS{\mathcal{HS}}
\def\hs{\mathcal{HS}}

\def\R{\mathbb{R}}

\def\Cc{\mathcal{C}}

\def\Yc{\mathcal{Y}}
\def\Ycov{\mathcal{Y}^{cov}}
\def\Ysep{\mathcal{Y}^{sep}}

\def\Dc{\mathcal{D}}
\def\Lc{\mathcal{L}}

\def\A{\mathbb{B}}
\def\B{\mathbb{B}}

\DeclareMathOperator{\Div}{div}

\def\lan{\left\langle}
\def\ran{\right\rangle}

%% file: title.tex
%\title{The extended adjoint state and nonlinearity in correlation-based\\ passive imaging}

% Sets running headers as well as PDF title and authors
\headers{Correlation-based passive imaging}{}

% Title
\title{The extended adjoint state and nonlinearity \\in correlation-based passive imaging\thanks{Submitted to the editors DATE}
\funding{The author acknowledges support from the DFG through
Grant 432680300 - SFB 1456 (C04)}}

% Authors: full names plus addresses.
\author{Tram Thi Ngoc Nguyen\thanks{Max Planck Institute for Solar Systems Research -- Fellow Group Inverse Problems  -- Justus-von-Liebig-Weg 3, 37077 G\"ottingen,  Germany
  (\email{nguyen@mps.mpg.de}\url{})}}

%% file: content.tex
\begin{abstract}
This articles investigates physics-based passive imaging problem, wherein one infers an unknown medium using ambient noise and correlation of the noise signal. We develop a general backpropagation framework via the so-called extended adjoint state, suitable for any \blue{elliptic} PDE; crucially, this approach reduces by half the number of required PDE solves. Applications to several different PDE models demonstrate the universality of our method. In addition, we analyze the nonlinearity of the correlated model, revealing a surprising tangential cone condition-like structure, thereby advancing the state of the art towards a convergence guarantee for regularized reconstruction in passive imaging.
\end{abstract}

% REQUIRED
\begin{keywords}
  Inverse problems, PDEs, regularization, passive imaging, correlation, covariance operator, tangential cone condition
\end{keywords}

% REQUIRED
\begin{AMS}
  65M32, 65J22, 35R30
\end{AMS}

\section{Introduction}\label{sec:intro}

The advent of the field of \emph{passive imaging} has profoundly impacted how we process measurement noise. Passive imaging refers to imaging with random uncontrolled excitation, in contrast to an active controlled source. By cross-correlating noise recorded at two sensors, one can infer underlying physical mechanisms, for example, wave travel-time between sensors, extracting much details about the intervening medium. 

Historically, the idea of exploiting cross-correlation of noise signals was founded by Duvall et al.~in 1993 for the study of helioseismology \cite{Duvall93}.
Here, solar oscillations are stochastically excited by the turbulent motions in the convection zone \cite{Dalsgaard, GizonBirchReview}.  The cross-correlation of the Dopplergram in the time or frequency domain (cross-covariance) provides much richer information and has yielded great success in imaging the Sun, especially the active region on its far side \cite{Braun_2001, Yang23}. Information is yielded via the Green's function -- the wave solution of a point source -- due to the celebrated Helmholtz-Kirchhoff identity, which proves that correlation of two Green’s functions on a surface 
equals the symmetrized Green’s function in the inhomogeneous interior. This information can then be used to characterize the unknown medium of the interior. We refer to \cite{Papanicolaou16} for a comprehensive overview on passive imaging with ambient noise. 

Passive imaging techniques have also seen great success in Earth seismology, often referred to as seismic interferometry, since illuminating sources are often rare and rather expensive. Correlation-based seismology using only ambient noise sources offers an alternative for, among others, exploration geophysics \cite{ShapiroCampillo04,sabra05},  travel-time retrieval \cite{Claerbout68},  embedded reflectors \cite{Borcea_2003}, volcano monitoring \cite{Brenguier11} and petroleum field monitoring \cite{Draganov13}. Beyond seismology, there are many new emerging areas using correlation data, such as passive synthetic aperture radars \cite{Borcea_2012},  optical speckle illumination \cite{Divitt20}, telescope imaging \cite{HELIN2018132}, aeroacoustics \cite{Raumer} as well as structural health monitoring \cite{Sabra07}.

Recently, \cite{Bjoern23} pioneered the combination of correlation data with iterative \emph{regularization methods} for solar flows and sound speed inference in local helioseismology.
Indeed, regularization theory provides a universal tool with a wealth of qualitative inversion techniques that can be applied to any imaging or inverse problem. Popular algorithms include Tikhonov regularization \cite{Tikhonov}, the factorization method, regularizing
filters \cite{Kirsch}, Landweber-type iteration, Newton-type methods \cite{Kaltenbacher97} and the Kaczmarz scheme, to name just a few.  These solution methods particularly emphasize designing special regularizers, such as priors, low-rank techniques, stopping rules and more, that can effectively address the \emph{ill-posedness} nature of inverse problems. 
For an insight into regularization theory, we refer to the seminal book \cite{EHNBuch}.

\begin{figure}[H] 
\centering
\includegraphics[width=13cm]{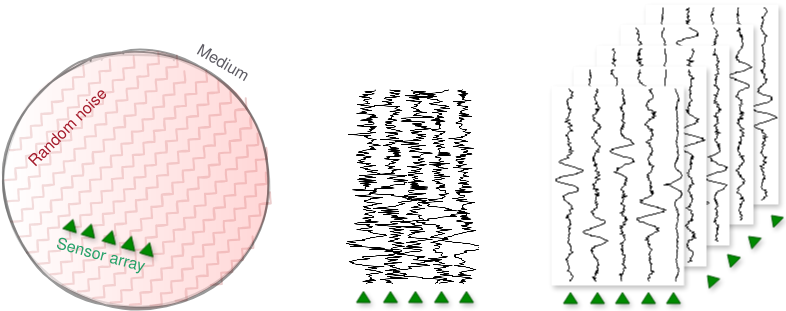}
\caption{Demonstration. Left: Noise source distributed through medium. Middle: Noise signal recorded at sensors. Right: Correlation between pairs of sensors.}
\end{figure}

The challenge of correlation-based passive imaging is twofold. First, by correlating data, we \enquote{square} the dimensionality of the problem. Thus, it is of high priority to derive backpropagators or \underline{adjoint operators}, which form the core of any inversion process, in an efficient and systematic manner. Secondly, correlation  measurement is in fact a nonlinear transformation, exacerbating \underline{nonlinearity} of the model; if the model problem is too nonlinear, any iterative reconstruction method may fail to converge. Overcoming these two difficulties is the central aspect of this study. We take a multidisciplinary approach, borrowing tools from stochastic processes, elliptic partial differential equations (PDEs) and regularization theory to obtain the key results.

By exploiting the covariance structure of data, in particular symmetry and the tensor-product formulation, we characterize the backpropagator via the so-called \enquote{extended adjoint state}. Surprisingly, this exposes a parallel to the classical linear measurement scenario. The extended adjoint state is, however, defined on a squared domain, with Sobolev-Bochner regularity. Crucially, this characterization allows for solving much fewer PDEs in the inversion, and is amenable to parallel computation. In addition, we demonstrate the adjoint computation in several benchmark applications, illustrating universality of the proposed routine. \blue{Our framework is general and valid for any linear elliptic
PDE, and extends to nonlinear equations, requiring neither explicit computation of Green’s functions, nor relying on their properties}. This forms our first key contribution.

The tools to characterize nonlinearity, particularly for ill-posed problems, are called nonlinearity conditions, and include range invariance \cite{Barbara-RangeInv24}, convexity, the Polyak-Lojasiewicz condition \cite{nguyen24} and the tangential cone condition \cite{scherzer95}. These conditions guarantee convergence and stability of various classical regularization methods \cite{KalNeuSch08}. This article focuses on the latter, seeking to leverage the well-established power of the \emph{tangential cone condition} (TCC). However, verifying this condition, even for specific examples and linear measurements, can be very technical \cite{TCC21}. Our findings reveals that there is a TCC-like structure in the correlated model. To the best of our knowledge, the tangential cone condition has yet to be investigated for the case of nonlinear measurements; this forms our second main contribution.

\subsubsection*{Outline} 
The article is organized as follows: Section \ref{sec:ip-formulation} formulates passive imaging as a PDE-based parameter identification problem with covariance measurements. Section \ref{sec:sensitivity-adjoint} studies differentiability of the model's forward map, introduces the extend adjoint state and derives the backpropagator. Based on this, Section \ref{sec:abc-problems and discussion} computes backpropagators for several applications and discusses different extensions. Section \ref{sec:nonlinearity} analyzes the extreme nonlinearity of the correlated model and reveals the TCC-like structure. Finally, Section \ref{sec:conclusion} concludes our findings and suggests future research.

\subsection*{Notation}

The following notational conventions are used throughout this article:
\begin{itemize}[label=$\circ$,leftmargin=0.3cm]
\item The notation $\|\cdot\|_{X\to Y}$ %and $\|\cdot\|_{\Lc(X,Y)}$ 
is  the operator norm, and $\|\cdot\|_{\HS(X,Y)}$ is the Hilbert-Schmidt norm. Recall that the Hilbert-Schmidt norm of a linear bounded operator $K$ between Hilbert spaces $X$, $Y$ is defined as the squared sum of its singular values $\sigma_i$, that is, $\|K\|_{\HS(X,Y)}=(\sum_i \sigma_i(K)^p)^{1/p}$ with $p=2$; %Clearly, Hilbert-Schmidt operators are compact; 
the  generalization to $p\in [1,\infty]$ defines the so-called Schatten $p$-classes.

\item The Sobolev-Bochner space $L^p(\dom;X)$ with $p=2$ and a Banach space $X$ is a Banach space consisting of square-integrable functions $y:\Omega\mapsto X$ with the norm defined in \eqref{Ysep}. The generalization to spaces with $p\in[1,\infty]$ and  to spaces with \enquote{time}-derivatives on the “time”-interval $\dom=[0,T]$ can be found in \cite{Roubicek}.

\item We use $X\embed Y$ for general continuous embeddings in Sobolev, Bochner or Hilbert-Schmidt spaces, with norm denoted by $C_{X\to Y}$. For compact embeddings, we instead write $X\compt Y$.

\item A triple $(X,H,X^*)$ consisting of a Hilbert (pivot) space $H$ and a dense, continuously embedded subspace $X$ with dual $X^*$ is referred to as a \emph{Gelfand triple}. The following Gelfand triples in Sobolev spaces \cite{Evans}, Hilbert-Schmidt spaces \cite{MeiseVogt92} and Bochner spaces \cite{Roubicek} will be frequently used:
\begin{gather*}
X\embed H\embed X^*\\
\HS(X^*,X)\embed \HS(H)\embed \HS(X,X^*)\\
L^2(\dom;X) \embed L^2(\dom;H)\embed L^2(\dom; X^*).
\end{gather*}

\item For generic bounded linear operators $A$, we use $A^\star$ to refer to  Banach space adjoint, whereas $A^*$ is the  Hilbert space adjoint, both with respect to the canonical complex dual parings/inner products. Additionally, $I_X:X\to X^*$ denotes the Riesz isomorphisms between a general Hilbert space $X$ and its dual $X^*$; $(\cdot,\cdot)$ denotes the inner product on a Hilbert space, while $\lan \cdot,\cdot\ran$ refers to the dual paring between Banach spaces.

\item For all that remains, $X$ will denote the real parameter space, $U$ is the state space, $W^*$ is the source space,  and $Y$, $\Yc$ are the data spaces. Each of $X$, $U$, $W$ together with their respective duals and a suitable pivot space is assumed to form a Gelfand triple.

\item Finally, $\overline{(\cdot)}$ is the complex conjugate,  $i$ is the  imaginary unit, and $\Re,\Im$ are  the real and imaginary parts. By $\Bc^X_\rho(\theta)$ we mean the ball in $X$ centered at $\theta$ with radius $\rho>0$.
\end{itemize}

\section{Inverse problems with correlation data}\label{sec:ip-formulation}

We begin by casting passive imaging into a unified framework of parameter identification inverse problem governed by linear PDEs. In this manner, we develop our analysis for a broad class of problems, covering several benchmark examples. As data is prerequisite for inverse problems, we will then thoroughly discuss the role of measurement operators, from traditional linear measurements to the nonlinear covariance measurement employed in passive imaging. 

\subsection{Parameter identification in PDEs}

\subsubsection*{Parametric PDEs in general setting} We study the inverse problem of determining an unknown real-valued parameter  $\theta$ in linear elliptic PDEs. The model parameter $\theta\in X$ and the PDE solution (or state) $u\in U$  are linked via the affine bilinear model $\B$ described as
\begin{align}\label{A}
\A(\theta,u) =:D(\theta)u=f \qquad\text{with } D(\theta)\in\Lc(U,W^*)\qquad\qquad\qquad
\end{align}
in which $D(\theta)$ denotes the parametric linear bounded differential operator, possibly including  boundary conditions. On the right hand side, $f\in W^*$ plays the role of a source term specified on a smooth domain $\Omega$. Moreover, the model $D$ depends on the model parameter $\theta$ in an affine linear manner, that is,
\begin{equation}\label{B}
\begin{split}
\B(\theta,u)=:Ku+B(u)\theta \qquad\text{with }& K\in\Lc(U,W^*), %\\ B(u)\in\Lc(X,W^*), \\
\quad B\in\Lc(U,\Lc(X,W^*))
\end{split}
\end{equation}
with linear bounded operators $K$ and $B$ mapping between corresponding spaces. This formulation covers a wide range of applications, e.g.~the Helmholtz acoustic wave equation $\B(\theta,u) = \Delta  u + \theta u$ with $D(\theta)=\Delta + \theta$, $K=\Delta$, $B=\text{Id}$ and the unknown squared wavenumber $\theta$; many more examples are discussed in Section \ref{sec:abc-problems and discussion}.
We will generally consider a Sobolev space framework, where the source space $W^*$ is taken to be weaker than the state space $U$; this setting is conventional for elliptic PDEs.

Throughout, $\A'_\theta$ and $\A'_u$ denote the partial derivatives; as $\A$ is affine bilinear, we have
\begin{align}\label{A'}
\A'_u(\theta,u) = D(\theta)\in\Lc(U,W^*),\qquad \A'_\theta(\theta,u) = B(u)\in\Lc(X,W^*).
\end{align}
Importantly, we assume the linear differential operator $D(\theta)$ with $\theta$ in some domain $\Dc$ 
to be boundedly invertible. Thus, its inverse is also linear and bounded, i.e.
\begin{align}\label{Dinv}
D(\theta)^{-1}\in\Lc(W^*,U) \quad \text{for }\theta\in\Dc.
\end{align}
This means that for a given model parameter $\theta\in \Dc$ and a source term $f\in W^*$, there exists a unique $u\in U$ satisfying \eqref{A}; in other words, the PDE \eqref{A} is uniquely solvable. Hence, one can define the parameter-to-state map (or PDE solution map) by
\begin{align}\label{S}
S_f:\Dc(\subseteq X)\to U, \quad S_f(\theta):=D(\theta)^{-1}f=u. % \quad\text{s.t }\A(\theta,S(\theta))=0.
\end{align}
Clearly, this map is nonlinear in the unknown $\theta$, despite the PDE being linear in $u$.

\subsubsection*{The role of measurement operators} For the inverse problem of identifying $\theta$, we need additional data; we denote it by $y$. Traditionally, one takes measurements of the state $u$ on the full domain, on the boundary (trace) or on a subdomain $\dom'$ of $\dom$, that is, 
\begin{align}\label{data-linear}
y=M u\quad M: U\to Y \quad\text{with } M=\text{Id}, \, \text{Trc} \text{ or } \chi_{\dom'}.
\end{align}
Here, $Y$ is the data space, which is generally much weaker than the PDE solution space $U$. Importantly, this measurement operator $M$ is \underline{linear} in $u$. Linear, compact observation is the standard scenario in many benchmark inverse problems, such as elastography with full measurements \cite{HubmerScherzer:TCC18, Gesa24}, inverse obstacle and medium scattering with trace observation \cite{bookChadanColtonPaivarintaLassi, Kirsch, HaddarMonk_2002}, electrical impedance tomography with Dirichlet-to-Neumman trace map \cite{Borcea_2002, Harrach21}; we refer to \cite{bookIsakov} for a comprehensive overview on measurement strategies for inverse problem with PDEs.

This leads to the construction of the parameter-to-observable map $F$ by composing the measurement operator and the parameter-to-state map, that is, with
\begin{align}\label{Flinear}
F: \Dc(\subseteq X)\to Y \quad F:= M\circ S_f
\end{align}
serving as the forward operator. The model $F$ constructed in this way is known as the reduced formulation. We remark that there are also other ways to build to the forward map, such as the all-at-once \cite{Kaltenbacher:17, Nguyen:19} and bi-level \cite{nguyen24, nguyen-seqbi} formulations.

The composition structure \eqref{Flinear} exposes an important fact: nonlinearity of the model is inherited directly from that of the well-posed parameter-to-state map $S_f$. On the other hand, its ill-posedness is induced by the compactness of the linear measurement $M$; see Appendix \ref{appendix}. Studying ill-posed inverse problems, especially nonlinear problems, requires careful treatment, even with full data $M=\text{Id}$. Indeed, the challenge lies in constructing regularization methods that can handle both ill-posedness and convergence under nonlinearity.

\subsection{Correlation measurement}

Until this point, we have assumed the source $f$ to be known beforehand, as this is required to establish the parameter-to-state map $S_f$ in \eqref{S}. However, in the context of passive imaging discussed in Section \ref{sec:intro}, excitation is due to turbulent motions or ambient noise, which we do not have direct access or control to; in this setting, $f$ is referred to as a \emph{passive source}.

\subsubsection*{Passive source} 
Such an $f$ is a realization of normally distributed random noise with zero mean and known covariance. The generated signal $u$ is consequently also normally distributed with zero mean; this follows by linearity and well-posedness of the PDE. As such, the contribution of $u$ to e.g.~iterative reconstruction algorithms such as \eqref{Lanweder}, where it enters via addition of a linearly backpropagated term, is $0$ in expectation, and could thus be viewed as uninformative. A more suitable approach is hence to collect raw data, then perform preprocessing by time auto-correlation, equivalently by taking the auto-covariance in the frequency domain. Here, we focus on the latter, keeping the term correlation as conventional terminology. In this way, the source covariance will be translated to the covariance of the observed state.

\subsubsection*{Nonlinear covariance measurements} More precisely, the auto-covariance of the state at two locations $x$, $x'$ in the spatial domain $\dom$ is defined by
\begin{align*}
\mathbb{C}\text{ov}[u,u](x,x')=\E[u(x)\overline{u(x')}] \qquad x,x'\in\dom
\end{align*}
where $\overline{(\cdot)}$ is the complex conjugate, employing the zero mean information $\E[u]=0$.
In practice, one usually estimates the covariance $\mathbb{C}\text{ov}$ by the sample covariance via
\begin{align}\label{sample-cov}
\mathbb{C}\text{ov}[u,u](x,x')\approx
\cov[u,u](x,x'):=
\frac{1}{J}\sum_{j=1}^J u_j(x)\overline{u_j(x')}% =\frac{1}{J}\sum_{j=1}^J S_{f_i}(\theta)(x)\overline{S_{f_i}(\theta)}(x')
\end{align}
for all $x$, $x'\in\dom$, where the right-hand side is interpreted as the $J$-average of correlated sample states. Since each $u_i$ is  the unique state generated from a sample source $f_i$ through the parametric PDE \eqref{S}, we write
\begin{align}\label{ip-data2}
\cov[u,u]=\frac{1}{J}\sum_{j=1}^J S_{f_j}(\theta)\,\overline{S_{f_j}(\theta)}.
%\cov[u,u]=\frac{1}{J}\sum_{j=1}^J D(\theta)^{-1}f_j\,\overline{D(\theta)^{-1}f_j
\end{align}
The $J$-sample covariance is an essential concept, as it will act as the input data for the parameter inversion. At this point, we make an importation observation.

\begin{remark}
While the standard measurement strategy \eqref{data-linear} is linear, the covariance measurement  \eqref{ip-data2} is a \underline{nonlinear} transformation of $u$. This amplifies nonlinearity in the parameter-to-observable map $F$; in particular, nonlinearity in the unknown parameter  $\theta$ is \enquote{squared} after taking correlation.
\end{remark}

\subsection{Passive imaging with covariance data}{\label{sec:ip}

We can now formulate our inverse problem. First, for any $ u\in (U\subseteq) Y$, the sample covariance operator can be viewed as a linear operator belonging to $\HS(Y^*,Y)$, the space of Hilbert-Schmidt operators, with its action given as 
\begin{align}\label{cov-inS}
\cov[u,u]: \varphi\in Y^* \mapsto \frac{1}{J}\sum_{j=1}^J \lan\varphi,u_j\ran_{Y^*,Y}u_j \in Y
\end{align}
via the dual paring $\lan\cdot,\cdot\ran$, where $u_i\in Y$ is sample of the $Y$-valued random variable  $u$. The inclusion $\cov[u,u]\in\HS(Y^*,Y)$ is clear due to finiteness of the singular system. 

We consider the practical choice $Y=L^2(\dom)$ for the observation space, as measurement is always contaminated by noise and generally does not preserve smoothness of $u\in U$. In this case, the Hilbert-Schmidt integral operator $(K\varphi)(x)=\int_{\dom\times \dom} k(x,x')\varphi(x')\,dx'$  has a one-to-one correspondence to its $L^2$-kernel, with $\|K\|_\hs=\|k\|_{L^2(\dom\times \dom)}$. By abuse of notation, we will therefore not distinguish between $\cov$ as an operator and its kernel, employing the same notation for both. We thus focus on the kernel formulation and define the covariance measurement operator
\begin{equation}\label{cov}
\begin{split}
\Cc:U^J\to L^2(\doms)=:\Yc, \qquad &\Cc(u):=\cov[u,u] \quad\text{with } \cov \text{ as in }\eqref{ip-data2}.
\end{split}
\end{equation}
Together with the vectorial parameter-to-state map
\begin{align}\label{Svec}
\Sv_f: \Dc(\subseteq X)\to U^J, \quad \Sv_f:= [S_{f_1}, S_{f_2}\ldots S_{f_J}] \quad\text{with } S_{f_j} \text{ as in }\eqref{S},
\end{align}
we construct the forward operator in the passive imaging problem as
\begin{align}\label{F-inS}
F:  \Dc(\subseteq X) \to \Yc, \qquad F(\theta):=\Cc\circ \Sv_f(\theta).
\end{align}
While this appears similar to \eqref{Flinear}, the nonlinear measurement operator $\Cc$ replaces the linear measurement operator $M$, and the data space is significantly higher-dimensional. 

A key complicating assumption is that we do not have access to individual sources $f_j\in W^*$, but only indirect information in the form of the sample source covariance. By \eqref{S} and \eqref{cov-inS}, the state sample covariance operator can be expressed via the source sample covariance by
\begin{align}\label{F_incov}
&\cov[u,u] %= \cov[D(\theta)^{-1}f,D(\theta)^{-1}f] =\sumj \lan \cdot,D(\theta)^{-1} f_j \ran_{V^*,V} D(\theta)^{-1}f_j   
=\sumj \lan (D(\theta)^{-1})^\star\cdot, f_j \ran_{W,W^*} D(\theta)^{-1}f_j 
= D(\theta)^{-1}\cov[f,f](D(\theta)^{-1})^\star
\end{align}
with $\cov[f,f]\in \HS(W,W^*)$ and the Banach space adjoint $(D(\theta)^{-1})^\star\in\Lc(U^*,W)$. Again, due to the one-to-one correspondence, the kernel $\cov[u,u]$ can be computed from the covariance operator. The expression \eqref{F_incov} provides an alternative for \eqref{F-inS} and is valid also for the exact covariance, i.e.~with $\mathbb{C}\text{ov}$ in place of $\cov$. Furthermore, if $D(\theta)^{-1}$ has an integral representation, e.g.~via the Green's function $G_\theta$, then the covariance kernel can be expressed via $\cov[f,f]$ by
\begin{align}\label{F_inG}
\cov[u,u](x,x')\stackrel{D(\theta)^{-1}f= \langle G_\theta, f\rangle}{=}\int_\dom\int_\dom G_\theta(x,s)\overline{G_\theta(x',s')}\cov[f,f](s,s')\,ds\,ds'.
\end{align}
 This formulation also holds for the exact covariance, in which the double integral can  be further collapsed to a single integral under the assumption of the source being spatially uncorrelated, in the sense that $\mathbb{C}\text{ov}(f,f)(s,s')=\Pi(s)\delta (s-s')$. In this article, we consider the general noise source model $\cov[f,f]\in \HS(W,W^*)$. Taking a stochastic viewpoint, the model \eqref{cov}-\eqref{F-inS} is equivalent to observing $J$ i.i.d.~samples drawn from the state distribution, then assemble the sample covariance \eqref{cov-inS}. As a stochastic estimator of the true state covariance, it is unbiased and consistent, and is known to follow a Wishart distribution due to Gaussianity of the state; for further notes on multivariate stochastics, we refer to \cite{Muirhead82}.
 
The formulations  \eqref{ip-data2}, \eqref{F_incov}, \eqref{F_inG} again highlight the role of the covariance measurement in squaring nonlinearity of the forward map $F$, which we will henceforth also refer to as the parameter-to-covariance map. We now state our passive imaging problem as 
\begin{align}\label{ip}
\text{find }\theta: \quad F(\theta)=y \quad\text{with}\quad F\text{ in }\eqref{F-inS},\text{ given $\cov[f,f]$ and covariance data $y$}.
\end{align}
From now on, we will drop the subscript $f$ in the parameter-to-state maps $S_f$ and $\Sv_f$, keeping their dependence on $f$ via \eqref{S} in mind.

\section{Sensitivity and adjoint state}\label{sec:sensitivity-adjoint}

Differentiability and adjoint derivation are of fundamental importance for any gradient-based inversion scheme. As an example, consider the Landweber (regularized gradient descent) method for solving a general nonlinear equation $F(\theta)=y$ given a noise-corrupted version $y^\delta$ of the true data $y$. The iteration reads as
\begin{align}\label{Lanweder}
\theta_{k+1} = \theta_k - \mu_k F'(\theta_k)^*\(F(\theta_k)-y^\delta\) \quad k=1\ldots k_\text{max}.
\end{align}
Besides identifying a suitable step size $\mu_k$ and stopping index $k_\text{max}$, this method requires evaluations with linearization $F'(\theta_k)$ and Hilbert space adjoint of derivative $F'(\theta_k)^*$. These two key ingredients appear in all gradient-based regularization methods, such as Nesterov-Landweber, Gauss-Newton, Levenberg-Marquardt and so on; see \cite{KalNeuSch08}.

Realized in our setting, the forward operator $F$ is the composition of the nonlinear parameter-to-solution map $S$ with the measurement map, which is linear in the non-correlated case and nonlinear for the case of correlation data. Section \ref{sec:linear} investigates differentiability and adjoint for linear measurements, which will lay the groundwork for  covariance measurements in Section \ref{sec:nonlinear}.

\subsection{Linear measurements}\label{sec:linear}

\begin{lemma}[Linearized state]\label{lem:linearized-state}
The parameter-to-state map $S$ in \eqref{S} is Fr\'echet differentiable, with derivative
\begin{align}\label{linearized-state}
S'(\theta):\Dc (\subseteq X)\to U, \quad S'(\theta)h=\phi \quad\text{where $\phi$ uniquely solves}\quad
D(\theta)\phi=-B(u)h%=-L(h)u
\end{align}
and with the state $u=S(\theta)$; we refer to $\phi$ as the linearized state.
\end{lemma}

\begin{proof}
The proof is a direct application of the implicit function theorem to the bilinear equation $\A(\theta,u)=f$ in \eqref{A}. In particular, with the partial derivatives of $\B$ in \eqref{B}, the implicit function theorem yields
\begin{align*}
0=\A'_\theta(\theta,u)+\A'_u(\theta,u)S'(\theta)=B(u) + D(\theta) S'(\theta).
\end{align*}
as claimed in \eqref{linearized-state}. Invertibility of $D(\theta)\in\Lc(U,W^*)$ for $\theta\in\Dc$ is assumed in \eqref{S}, and the source $-B(u)h\in W^*$ guarantees unique existence of the linearized state $\phi=-D(\theta)^{-1}B(u)h\in U$.
\end{proof}

We now derive the adjoint for the linearized parameter-to-state map via  solution of the adjoint PDE; this solution is called the adjoint state.

\begin{lemma}[Adjoint state]\label{lem:adj-full}
The Hilbert space adjoint of the linearized parameter-to-state map is
\begin{align}\label{adj-full}
S'(\theta)^* : L^2(\dom,\C)\to X, \quad S'(\theta)^* y  = -I_X\Re B(u)^\star \psi \quad{s.t.}\quad  D(\theta)^\star\psi=  y
\end{align}
with the adjoint state $\psi$ and Riesz isomorphism $I_X:X^*\to X$.
\end{lemma}

\begin{proof} 
First of all, since the parameter space $X$ is real, we impose on the data space $L^2(\dom, \C)$ the real-valued inner product
\begin{align}\label{real-inprod}
(f,g)^\R_{L^2}:=\(\re{f},\re{g}\)_{L^2}+\(\im{f},\im{g}\)_{L^2}=\re{(f,g)_{L^2}},
\end{align}
that is, the real part of the canonical complex-valued inner product.
 
By invertibility of $D(\theta)\in\Lc(U,W^*)$, thus of its Banach space adjoint $D(\theta)^\star\in\Lc(W,U^*)$, there exists for any $y\in L^2(\dom,\C)\subseteq U^*$ some $\psi\in W$ such that $D(\theta)^\star \psi=y$. Using the linearized state $\phi$ in \eqref{linearized-state}, the representation \eqref{adj-full}, we have, for any $h\in X$, $y\in Y$,
\begin{align*}
\( S'(\theta)h, y \)^\R_{L^2}&\stackrel{\eqref{real-inprod}}{=}\Re\lan S'(\theta)h, y \ran_{U,U^*} \stackrel{\eqref{linearized-state},\eqref{adj-full}}{=} \Re \lan \phi, D(\theta)^\star \psi\ran_{U,U^*} = \Re \lan D(\theta)\phi, \psi\ran_{W^*,W} \\
&\stackrel{\eqref{linearized-state}}{=} \Re\lan -B(u)h, \psi\ran_{W^*,W} 
=\Re \lan h, - B(u)^\star \psi\ran_{X,X^*}=\( h, -I_X\Re B(u)^\star \psi\)_X
\end{align*}
with the complex Banach space adjoint $B(u)^\star\in\Lc(W,X^*)$, although strictly speaking, this requires working with the complexification of the real-valued space $X$ and its dual, then computing the adjoint with respect to the canonical complex inner product, which typically leads to straightforward computations; see  Section \ref{sec:abc} for examples. 
Then, the real-valued Banach space adjoint with respect to the original, non-complex space $X$ follows simply by taking the real part $\Re$. 

In the last step above, the Gelfand triple property of $X\embed L^2(\dom)\embed X^*$ allows us to transfer the dual paring to the inner product via the isomorphism $I_X$. The proof is complete.
\end{proof}

\begin{remark}[Adjoint - linear measurement] For general linear measurement operators $M$ as in \eqref{Flinear}, one has $F'(\theta)^* = S'(\theta)^* M^*$ with $M^*: Y\to L^2(\dom)$ and  $S'(\theta)^*:L^2(\dom) \to X$ in Lemma \ref{lem:adj-full}. As two examples, $M=\text{Id}$ has the adjoint $M^*=\text{Id}$, while the restriction $M=\chi_{\dom'\subsetneq\dom}$ satisfies $M^*=\text{E}^0_{\dom'\to\dom}: L^2(\dom')\to L^2(\dom)$, which is the zero extension to the full domain. 
\end{remark}

We remark that the adjoint state sees widespread usage not only in PDE-based inverse problems, but also in optimal control \cite{Troeltzsch}, mainly in the context of linear observations. Hence, it is natural to ask whether there exists a similar concept for the correlated model. This line of thought will guide us towards the next section.

\subsection{Covariance measurements} \label{sec:nonlinear}

First of all, we recall from Section \ref{sec:ip-formulation} that covariance measurements live in  $\Yc=L^2(\doms)$. This is a canonical and conventional choice; however, it does not exploit the \enquote{covariance structure} of the data. We thus begin by better characterizing the internal structure in the data space $\Yc$, revealing in particular its summability and separability, paving the way for the adjoint derivation.

\begin{definition}

Define by
\begin{align}\label{Ycov}
\Ycov:=\left\{ y\in L^2(\doms)\,|\, \exists I<\infty: y(x,x')=\sumi y_i(x)\overline{y_i(x')} \text{ where } y_i\in L^2(\dom) \right\}
\end{align}
the cone in $\Yc$ consisting of functions on the covariance form, and by
\begin{align}\label{Ysep}
\Ysep:=\left\{ y\in L^2(\doms)\,|\, \exists I<\infty: y(x,x')=\sum_{i=1}^I p_i(x)q_i(x') \text{ where } p_i,\,q_i\in L^2(\dom) \right\}
\end{align}
the subspace of $\Yc$ consisting of functions that can be decomposed into a finite sum of products.

\end{definition}
\begin{remark}[Covariance and 
decomposable structure]\label{rem:cov-strucure} The space $\Ysep$ can be viewed as having a tensor product structure. Clearly, it contains the covariance cone $\Ycov$, with
\[\Ycov \quad\subset\quad \Ysep \quad\subset\quad \Yc = L^2(\doms).\]
In $\Ycov$, all elements are Hermitian symmetric, in the sense $y(x,x')=\overline{y(x',x})$; the sets $\Ycov$ and $\Ysep$ overall carry much more structure than the original data space $\Yc$.
\end{remark}

In the following, we introduce the notion of an \enquote{extended adjoint state} with two key properties. Firstly, differing from the standard adjoint state in Lemma \ref{adj-full}, the extended adjoint state is defined on the extended/squared domain $\doms$. Secondly, thanks to the setting $\Ysep$, the extended adjoint state will be shown to live in the Sobolev-Bochner space $L^2(\dom;W)$ with the norm  
\begin{align}\label{bochner-norm}
y:\Omega\mapsto W, \quad \|y\|_{L^2(W)}:=\(\int_\dom \|y(\cdot,x')\|^2_{W}\,dx'\)^{1/2}.
\end{align}
The Bochner space $L^2(\Omega;W)$ is a Hilbert space when $W$ is a Hilbert space, and  $L^2(\dom;W)\subseteq L^2(\dom;L^2(\dom))\cong L^2(\doms)$ for $W\subseteq L^2(\dom)$; see \cite{Roubicek}.

\begin{lemma}[Extended adjoint state]\label{lem:adj-extend} For any covariance data $y\in\Ysep$, there exists a decomposable function $\Psi\in L^2(\dom;W)$ which solves the extended adjoint PDE
\begin{align}\label{adj-extend}
\forall^\text{a.e.} x'\in\dom:\quad D_1(\theta)^\star \Psi = y.
\end{align}
Here, $D_1(\theta)^\star:=D(\theta)^\star$ in \eqref{A}, with the subscript denoting the differential operator acting on $x$, the first variable of $\Psi$. We call $\Psi$ the extended adjoint state.
\end{lemma}
\begin{proof}
As $y\in\Ysep$ by assumption, we have $y=\sumi y_i\overline{q_i}$, for some families $(y_i)_{i=1}^I$, $(q_i)_{i=1}^I\subset L^2(\dom)$, $I<\infty$; see \eqref{Ysep}. Arguing as in Lemma \ref{lem:adj-full}, we claim that for each source $y_i\in L^2(\dom)$, there exists a unique $\psi_i\in W$ satisfying the adjoint equation
\begin{align}\label{adj}
\exists!\, \psi_i=\psi_i(x)\in W: \quad D(\theta)^\star \psi_i = y_i.
\end{align}
This yields for a.e.~$x'\in\dom$ that
\begin{align*}
y(\cdot,x') = \sumi y_i\overline{q_i(x')} = \sumi \left[D(\theta)^\star \psi_i\right]\overline{q_i(x')}= D(\theta)^\star \(\sumi  \psi_i\overline{q_i(x')}\). 
\end{align*}
Matching \eqref{adj-extend}, this means that the extended adjoint state $\Psi$ takes the form
\begin{align}\label{adj-ext-1}
\Psi(x,x'):=\sumi  \psi_i(x)\overline{q_i(x')}, \qquad \psi_i\in W, q_i\in L^2(\dom);
\end{align}
moreover, it lives in the Sobolev-Bochner space \eqref{bochner-norm}, as
\begin{equation*}\label{adj-ext-2}
\begin{split}
\|\Psi\|^2_{L^2(W)}&\leq2^{I-1}\sumi\int_\dom \|\psi_i(\cdot) q_i(x')\|_W^2\,dx'=2^{I-1}\sumi\int_\dom \|\psi_i(\cdot)\|_W^2| q_i(x')|^2\,dx'\\
&=2^{I-1}\sumi \|\psi_i\|^2_W\|q_i\|^2_{L^2} <\infty.
\end{split}
\end{equation*}
We observe from \eqref{adj-ext-1} that $\Psi$ also has a decomposable structure; the proof is thus complete.
\end{proof}

\begin{remark}[Computing the extended adjoint state]\label{rem:adj-ext}
The structure of $\Psi$ suggests that it may be computed in two different ways. The first option is via \eqref{adj-ext-1}, which solves $I$ standard adjoint PDEs \eqref{adj-full} for $\psi_i, i=1\ldots I$, then takes the outer product with $\overline{q_i}$, and sums them up afterwards. Alternatively, via \eqref{adj-extend}, one can solve the standard adjoint PDE on each $x'$-slice/column of the data $y$. In this case, the number of PDE applications depends on the discretization dimension $N$ of $\dom$, and the decomposition of data $y$ into explicit $y_j$, $q_j$ is not needed. In both cases, the $I$ or $N$-times PDE solving processes can be done on parallel machines as they are independent, \blue{albeit accounting for memory limitations on the degree of parallelization.}
\end{remark}

\begin{figure}[H] 
\centering
\includegraphics[width=13.5cm]{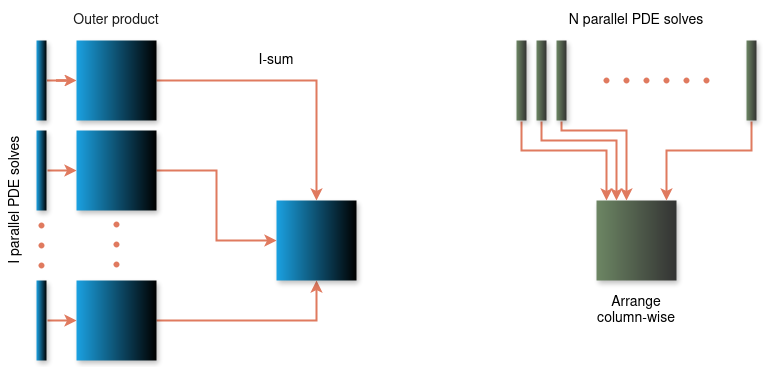}
\caption{Extended adjoint state computation. Left via \eqref{adj-ext-1}: parallel PDE solution  $\psi_i, i=1\ldots I$ --  outer product with $q_i$ -- take $\sum_i^I$. Right via \eqref{adj-extend}: parallel PDE solution $\Psi(\cdot,x'_n), n=1\ldots N$ -- collect and arrange in column-wise. }
\end{figure}

We now turn our attention to the cone $\Ycov$ for the main result.

\begin{theorem}[Adjoint - covariance measurement]\label{theo:adj-cov}
The Hilbert space adjoint of the linearized parameter-to-covariance map is
\begin{equation}\label{adj-cov}
\begin{split}
&F'(\theta)^*: \Ycov\to X, \quad F'(\theta)^*y=-2 I_X \Re\int_\dom B^{cov}_1(u,x')^\star  \Psi(\cdot,x')\,dx' \\[1ex]
&\text{with }\ B^\text{cov}_1(u,x'):=B(\cov(u,u)(\cdot,x'))\in\Lc(X,W^*)\quad \text{for a.e. } x'\in\dom, \text{ all } u\in U
\end{split}
\end{equation}
and $\Psi$ is the extended adjoint state in Lemma \ref{lem:adj-extend}.
\end{theorem}

\begin{proof}
With the state $u_j=S_{f_j}(\theta)$ and the linearized state $\phi_j=S_{f_j}'(\theta)h$ in Lemma \ref{lem:linearized-state}, one differentiates $F$ via the chain rule as
\begin{align}\label{linearization}
F'(\theta)h = (\cov\circ \Sv)'(\theta)h = \sumj u_j\overline{\phi_j} + \phi_j \ubj.
\end{align}
For any $y\in\Ycov$, we now express the $\Yc$-inner product as
\begin{align*}
\(F'(\theta)h, y\)_{\Yc}&=\sumj \int_\dom\int_\dom \( u_j(x)\overline{\phi_j(x')} + \phi_j(x) \overline{u_j(x')} \) \overline{y(x,x')}\,\,dx\,dx'\\
&=\sumj \int_\dom\int_\dom  \overline{\phi_j(x)}u_j(x')\overline{y(x',x)}\,\,dx\,dx'+ \int_\dom\int_\dom  \phi_j(x) \overline{u_j(x')}\,\overline{y(x,x')}\,\,dx\,dx',
\end{align*}
swapping the variables $x$ and $x'$ in the first double integral. Since $y\in\Ycov$, $y$ is Hermitian symmetric with $\overline{y(x',x)}=y(x,x')$; see also Remark \ref{rem:cov-structure}. We thus further simplify
\begin{align*}
\(F'(\theta)h, y\)_{\Yc}%&=\sumj \int_\dom\int_\dom  \overline{\phi_j(x)}u_j(x')y(x,x')\,\,dx\,dx'+ \int_\dom\int_\dom  \phi_j(x) \overline{u_j(x')}\,\overline{y(x,x')}\,\,dx\,dx'\\
&= 	2\Re\sumj \lan \phi_j(\cdot) \overline{u_j(\cdot\cdot)}, y(\cdot,\cdot\cdot)\ran.
\end{align*}

In the next step, we write
$D_1(\theta)\Psi=D(\theta)\Psi$ for clarity, where the subscript again indicates that the operator $D(\theta)\in\Lc(U,W^*)$ acts on $\Psi(\cdot,x')$. 
In particular, using the extended adjoint state $\Psi$ in Lemma  \ref{lem:adj-extend}, we express the inner product for any $h\in X$ and $y\in\Ycov$ as
\begin{align}\label{cov-adj-compute} 
\(F'(\theta)h, y\)_{\Yc}&\stackrel{\text{Lem. }\ref{lem:adj-extend}}{=}2\Re\sumj \Big\langle \phi_j(\cdot) \overline{u_j(\cdot\cdot)}, D_1(\theta)^\star \Psi(\cdot,\cdot\cdot) \Big\rangle
=2\Re\sumj\Big\langle D_1(\theta)\phi_j(\cdot) \overline{u_j(\cdot\cdot)},  \Psi(\cdot,\cdot\cdot) \Big\rangle \nonumber\\
&\stackrel{\text{Lem. } \ref{lem:linearized-state}}{=} 2\Re\sumj \Big\langle -B(u_j)h \overline{u_j(\cdot\cdot)}, \Psi(\cdot,\cdot\cdot) \Big\rangle \nonumber\\
&\stackrel{B\text{ linear}}{=} 2\Re\Big\langle -B\Big(\sumj u_j\overline{u_j(\cdot\cdot)}\Big) h, \Psi(\cdot,\cdot\cdot) \Big\rangle \stackrel{\text{Def. in }\eqref{adj-cov}}{=} 2\Re \Big\langle -B^{cov}_1(u,\cdot\cdot)h,  \Psi(\cdot,\cdot\cdot) \Big\rangle \\
&= 2\Re \Big\langle h, -B^{cov}_1(u,\cdot\cdot)^\star  \Psi(\cdot,\cdot\cdot) \Big\rangle.\nonumber
\end{align}
Well-definedness in the above expressions can be checked e.g.~by verifying in the second equality that $\Psi\in L^2(\dom;W)$ by Lemma \ref{lem:adj-extend}, together with $D_1(\theta)\phi_j\overline{u_j}\in L^2(\dom;W^*)$ by $D_1(\theta)\phi_j(\cdot)\in W^*$ and $u_j(\cdot\cdot)\in U\subseteq L^2(\dom)$. The fourth equality demonstrates the importance of affine bilinearity of the model $\A(\theta,u)$ in \eqref{A}, by which we can employ linearity of $B(\cdot)$ with respect to $u$.

As a side note, the Banach space adjoint $B(\cdot)^\star$ is conjugate linear in $u$, in the sense that
$
\lan \theta, B(u+\lambda v)^\star w\ran_{X,X^*}=\lan B(u)\theta+\lambda B(v)\theta, w\ran_{W^*,W}=\lan \theta, B(u)^\star w+ \overline{\lambda}B(v)^\star w\ran_{X,X^*}
$ holds
for any $u$, $v\in U$, $\theta\in X$, $w\in W$, $\lambda\in\C$. The adjoint $B^{cov}_1(u,x')^\star$ thus connects to $B(u)^\star$ via
\begin{align*}
B^{cov}_1(u,x')^\star= \(\sumj B(u_j) \overline{u_j(x')} \)^\star %=\sumj u_j(x') B(u_j)^\star 
= B\(\cov(u)(\cdot,x')\)^\star;
\end{align*}
in other words, $B^{cov}_1(u,x')^\star$ is for a.e.~$x'$ simply $B(u)^\star$ with $\cov(u,u)(\cdot,x')$ in place of $u$.

Returning to \eqref{cov-adj-compute} for the final step, since $h=h(x)$ is a one-variable function, collapsing the $x'$-dimension by integration then passing through the isomorphism $I_X$ similarly to Lemma \ref{adj-full}, we arrive at the inner product in $X$ as
\begin{align*}
\(F'(\theta)h, y\)_{\Yc} = 2\Re  \Big( h, -B^{cov}_1(u,\cdot\cdot))^\star  \Psi(\cdot,\cdot\cdot)\Big)_{L^2(\doms)}= \(h, -2 I_X\Re \int_\dom B^{cov}_1(u,x')^\star  \Psi(\cdot,x')\,dx' \)_X.
\end{align*}
This is the claim \eqref{adj-cov}, and the proof is thus complete.
\end{proof}
\bigskip

With the backpropagator derived, we close the section with some important remarks.

\begin{remark}[Covariance vs full measurement]

\begin{enumerate}[label=\roman*)]
\item The covariance adjoint \eqref{adj-cov} and full measurement adjoint \eqref{adj-full} share rather similar forms. In particular, \eqref{adj-cov} has $\cov(u,u)$ in place of $u$ in \eqref{adj-full}; they otherwise differ only by a factor $2$ and the presence of the integral, exposing an interesting connection between the two adjoints.

\item The extended adjoint equation \eqref{adj-extend} is in fact the standard adjoint equation \eqref{adj-full} solved column-wise/slice-wise in the data $y\in\Ysep$ (c.f. Remark \ref{rem:adj-ext}). In other words, the extended adjoint state $\Psi$ is a column-wise stack of several standard adjoint state $\psi$, reinforcing the connection between the two models.
\end{enumerate}
\end{remark}

\begin{remark}[Hermitian property]\label{rem:cov-structure}
\begin{enumerate}[label=\roman*)]
\item The real part $\Re$ appears in both full  and covariance adjoints. 
However, \eqref{adj-cov} is real-valued not due to the imposed real-valued inner product \eqref{real-inprod} in Lemma \ref{lem:adj-full}, but rather, due to the covariance structure itself. Indeed, for any $y$, $z\in\Ycov$, which are by construction Hermitian symmetric (see Remark \ref{rem:cov-strucure}), the real and complex inner products coincide, as
\begin{align*}
(y,z)_\Yc^\C%&=\frac{1}{2}\int_{\doms} y(x,x')\overline{z(x,x')}\,dx\,dx'+\frac{1}{2}\int_{\doms} y(x,x')\overline{z(x,x')}\,dx\,dx'\\
&=\frac{1}{2}\int y(x,x')\overline{z(x,x')}\,dx\,dx'+\frac{1}{2}\overline{\int y(x',x)\overline{z(x',x)}\,dx'\,dx} %\qquad(\in\R)\\
=(y,z)_\Yc^\R.
\end{align*} 
This justifies the view that the choice of the real inner product $(\cdot,\cdot)_\Yc^\R$ on $\Ycov$ is, in some sense, the natural one. 

\item The index $I$ in $\Ycov$, $\Ysep$ can differ from the sample size $J$ in $\cov$. In addition, Theorem  \ref{theo:adj-cov} only requires decomposable and symmetric data, allowing for any noisy observation $y^\delta$ of some exact but inaccessible data. Note that noisy measurements are taken prior to cross-correlating; hence, $y^\delta$ always has the prerequisite covariance structure despite the presence of noise corruption. In fact, it is sufficient to work in the cone $\Yc^\text{sep, sym}$:
\[
\Ycov \subset\quad \Yc^\text{sep, sym}:=\{y\in\Ysep\,|\,y\text{ Hermitian}\}\quad\subset \Ysep\subset \Yc.
\]
\end{enumerate}
\end{remark}

\begin{remark}[Computational \blue{aspects}]\,
\begin{enumerate}[label=\roman*)]
\item 
As discussed in Section \ref{sec:ip}, we do not have direct access to $f$, but rather an estimate of its covariance. This is reflected in Theorem \ref{theo:adj-cov}, as not $u$, but rather $\cov(u,u)$ emerges in the adjoint \eqref{adj-cov}; this sample covariance can be computed from $\cov(f,f)$  via \eqref{F_incov} or \eqref{F_inG}.

\item The characterization of the covariance adjoint can readily be integrated into any classical regularization method, e.g.~Landweber iteration \eqref{Lanweder}, yielding Algorithm \ref{Landweber-algorithm}.

The appearance of $\cov(u,u)$ in the adjoint \eqref{adj-cov} means that the forward-propagation \ref{S1} is re-used in the backpropagation \ref{S3} in Algorithm \ref{Landweber-algorithm}. In \ref{S2}, only one time PDE-block solver is called for the extended adjoint state. Alternatively, directly backpropagating from \eqref{F_incov} involves $D(\theta_k)^{-1}$ and $(D(\theta_k)^{-1})^\star$, requiring two blocks of PDE solves; in other words, our suggestion halves the number of required PDE compared to the direct approach.

\item
\blue{Numerical implementation depends on the exact discretization employed. 
For example, the slice-based computation is naturally carried out for collocation point-based discretization schemes, e.g.~finite difference. Meanwhile, finite element-based schemes require more subtle numerical analysis, and might benefit greatly from, e.g.~low-rank methods \cite{HalkoMartinssonTropp}, which enable rapid operator evaluation while avoiding full operator assembly, e.g.~for covariance operators \cite{aarset24}.} 

\end{enumerate}
\end{remark}

\begin{algorithm}[H]
\caption{Landweber iteration for passive imaging}
\label{Landweber-algorithm}
Initialization: $\theta_0$\\
While $k\leq k_\text{max}$: 
\begin{enumerate}[label= S.\arabic*]
\item\label{S1} Forward evaluate $F(\theta_k)=\cov(u,u)$ as \eqref{F_incov} or \eqref{F_inG}
\item\label{S2} Parallel computation of extended adjoint state $\Psi$ as \eqref{adj-extend} 
\item\label{S3} Backpropagate residual $F'(\theta_k)^*r$ as \eqref{adj-cov} using $\Psi$ and residual $r:=F(\theta_k)-y^\delta$
\item\label{S4} Update parameter $\theta_{k+1}$ as \eqref{Lanweder} with appropriate step size $\mu_k$
\item\label{S5} Check stopping rule $k_\text{max}$ according to a suitable a-priori or posterior rule 
\end{enumerate}
\end{algorithm}

\section{The $a, b, c$-problems and extension}\label{sec:abc-problems and discussion}

We now apply the developed adjoint state-based framework to several imaging applications. Our goal is to demonstrate that this framework yields backpropagators in explicit form via transparent computations. This also opens up various points for extension, which we discuss at the end of the section.

\subsection{Passive imaging for $a, b, c$  parameter inference}\label{sec:abc}  
Consider the following PDE on a smooth domain $\dom\subset\R^d, d\leq 3$  with the unknown physcial parameters $\theta:=(a,b,c)$:
\begin{equation}\label{abc}
\begin{cases}
&-\Div(a\nabla u)+ b\cdot \nabla u+icu = f \quad \text{in } \dom \\
&u=0 \quad \text{on } \partial\dom.
\end{cases}
\end{equation}
that we refer to as the \emph{$a,b,c$-problem}. Such equations have several real-world applications, for instance population evolution \cite{Malthus} for the $c$-problem, substance advection or transportation \cite{SocolofskyJirka} for the $b$-problem, sediment formation, groundwater filtration \cite{BanksKunisch} for the $a$-problem. A Hilbert space setting that ensures well-posedness of \eqref{abc} is
\begin{equation}\label{abc-space}
\begin{split}
&u\in U:=H_0^1(\dom), \qquad f\in W^*:=H^{-1}(\dom),\qquad \theta\in X:=(X_a,X_b,X_c),\\
&0<\underline{a}\leq a\in X_a:= H^2(\dom), \quad b\in X_b:=H^1(\dom,\R^d), \quad c\in X_c:=L^2(\dom),
\end{split}
\end{equation}
with appropriate smallness constraints on $b$ and $c$ in relation to $a$ to maintain uniform ellipticity, the key for PDE well-posedness. Positivity of $a$ along with these smallness conditions informs the choice of the domain $\Dc$ of the solution map $S$ in \eqref{S}. General Banach space frameworks can be found in \cite{LadyzhenskayaUraltseva, HoffmanWaldNguyen:2021}. Here, the homogeneous boundary condition is incorporated into the state space. 

\begin{proposition}\label{prop:abc}
Consider the passive imaging problem for $\theta=(a,b,c)$ in 
\eqref{abc}-\eqref{abc-space}.\\
The parameter-to-covariance map $F$ 
has its adjoint derivative applying to any $y\in\Ycov$  as 
\begin{align}\label{abc-adj}
F'(\theta)^* y 
=\begin{bmatrix}
-2I_{X_a}\Re\int_\dom \nabla_x\(\overline{\cov(u,u)}\)\cdot  \nabla_x \Psi\, dx' \\
-2I_{X_b}\Re\int_\dom \nabla_x\(\overline{\cov(u,u)}\)\Psi \, dx'\\
-2\Im\int_\dom \overline{\cov(u,u)}\Psi \, dx'
\end{bmatrix}\in (X_a,X_b,X_c)
\end{align}
with the \blue{Riesz} isomorphisms $I_{X_a}:X_a^*\to X_a$, $I_{X_b}:X_b^*\to X_b$. Here, the extended adjoint state $\Psi$ solves the  adjoint PDE
\begin{align}\label{abc-adj-eq}
\forall^\text{a.e.}x'\in\dom:
\begin{cases}
-\Div_x(a\nabla_x \Psi(\cdot,x'))\blue{-} \Div_x(b\Psi(\cdot,x')) -ic\Psi(\cdot,x') = y(\cdot,x') \quad\text{in } \dom\\
\Psi(\cdot,x')=0 \quad \text{on }\partial\dom
\end{cases}
\end{align}
where the subscript $x$ indicates differentiation in the $x$ variable.
\end{proposition}
\begin{proof} 

The proof will follow by application of Theorem \ref{theo:adj-cov}, which necessitates explicit analysis of the elements of \eqref{adj-cov}.

For \eqref{A}-\eqref{B}, one specifies the affine  bilinear model $\A\in\Lc(X\times U,W^*)$, the bounded invertible operator $D(\theta)\in\Lc(U,W^*)$ and the linear bounded operators $B(u)\in\Lc(X,W^*), K\in\Lc(U,W^*)$ via
\begin{align*}
&\A(\theta,u):=-\Div(a\nabla u)+ b\cdot \nabla u+icu = D(\theta)u = B(u)\theta+Ku\\
&D(\theta)=-\Div(a\nabla (\cdot))+ b\cdot\nabla +ic, \quad B(u)= \begin{bmatrix}
-\Div(\cdot\nabla u)\\ (\cdot)\cdot \nabla u\\ iu \end{bmatrix}, \quad K=0.
\end{align*} 
Using integration by part while taking into account the homogeneous Dirichlet boundary, we obtain the complex Banach space adjoints of the above as
\begin{align*}
D(\theta)^\star=-\Div(a\nabla (\cdot))\blue{-} \Div(b\,\cdot) -ic, \qquad B(u)^\star =\begin{bmatrix} \nabla \bar{u}\cdot\nabla \\\nabla \bar{u} \\-i\bar{u}\end{bmatrix},
\end{align*}
recalling that the state $u$ is complex-valued, while $a$, $b$, $c$ are real functions. 

Inserting $D(\theta)^\star$ into the extended adjoint PDE \eqref{adj-extend} yields the explicit form  \eqref{abc-adj-eq}, whose solution is the extended adjoint state $\Psi=\Psi(x,x')$. It is clear in \eqref{abc-adj-eq} that each $x'$-slide of $\Psi$ is the solution of the standard adjoint PDE (c.f.~Remark \ref{rem:adj-ext}).
With $\Psi$ and $B(u)^\star$, one can explicitly describe $B^\text{cov}_1(u)^\star=B(\cov(u,u))^\star$ in Theorem \ref{theo:adj-cov} by
\begin{align}
&B^\text{cov}_1(u))^\star \Psi
=\begin{bmatrix}
\nabla_x\(\overline{\cov(u,u)}\)\cdot \nabla_x \Psi \\
\nabla_x\(\overline{\cov(u,u)}\)\Psi\\
-i\, \overline{\cov(u,u)}\Psi
\end{bmatrix}.
\end{align}
As these steps provide all necessary insight into the elements in the expression \eqref{adj-cov}, we attain the backpropagator as \eqref{abc-adj} claims. 
\end{proof}

As an example, 
if $a$ and $b$ have vanishing boundaries, then one can choose for the isomorphisms $I_{X_a}=(\Delta^2)^{-1}$, $I_{X_b}= (-\Delta)^{-1}$ in the above. Proposition \ref{prop:abc} illustrates the transparency and universality of the proposed backpropagator in practice. Its computation is straightforward: one constructs the adjoint of the differential operators $D(\theta)$ and of $B(u)$; then inserts them into \eqref{adj-cov} of Theorem \ref{theo:adj-cov}.
 
\subsection{Discussion on extensions}\label{sec:discussion}

\begin{discussion}[General space settings]\label{dis-general-space}
Proposition \ref{prop:abc} provides the explicit adjoint formulation for \eqref{abc}
not only in the framework \eqref{abc-space}, but also in other general settings, requiring only minimal change, primarily in the Riesz isomorphisms $I_X=(I_{X_a},I_{X_b},I_{X_c})$. As an example, one could instead set
$U=H^2(\dom)\cap H_0^1(\dom)$, $W^*=L^2(\dom)$ and $(X_a,X_b,X_c)=(H^2(\dom), H^1(\dom), H^1(\dom))$. Comparing to \eqref{abc-space}, there is an increase in regularity, meaning smoother sources and parameters inducing smoother states. Lifted or shifted regularity is in fact a natural phenomenon in linear PDEs given sufficiently smooth domain $\dom$, see e.g.~\cite[Chapter 6.3]{Evans}.
\end{discussion}

\begin{discussion}[Other linear PDEs]
It is straightforward to adapt Proposition \ref{prop:abc} to other linear elliptic PDEs, e.g.~bi-Laplace equations; see Example \ref{ex:billap}.
\end{discussion}

\begin{discussion}[Nonlinear PDEs in the parameter] So far we have focused on bilinear PDEs in the form $\B(\theta,u)=D(\theta)u=Ku+B(u)\theta$. In fact, this analysis can be naturally transferred to equations that are nonlinear in $\theta$. 

In this case, linearization yields $\B'_\theta(\theta,u)$, as opposed to $B(u)$. In Lemma \ref{lem:linearized-state}, the linearized equation now takes the form $D(\theta)\phi=-\B'_\theta(\theta,u)h$. Consequently, in Lemma \ref{lem:adj-full}, the full measurement adjoint becomes $S'(\theta)^* y  = -I_X \Re\,\B'_\theta(\theta,u)^\star \psi$. In the end, the covariance adjoint in Theorem \ref{theo:adj-cov} is modified to $F'(\theta)^*y=-2 I_X \Re\int_\dom \B'_\theta(\theta,\cov(u,u))^\star  \Psi(\cdot,x')\,dx'$; see Example \ref{ex:billap}.
\end{discussion}

\begin{discussion}[Nonlinear PDEs in the state] 
The extension to nonlinearity in the state, however, is nontrivial. The argument on linearity in $u$ is essential for the proof of Theorem \ref{theo:adj-cov}. In addition, the source-to-state map $f\mapsto u$  in this scenario is also nonlinear, leading to several obstacles: the state covariance cannot be computed from the source covariance \eqref{F_incov}; stochastically, the distribution of $u$ is also not generally available, despite Gaussianity of $f$. 
\end{discussion}

\begin{example}\label{ex:billap}
We here study the bi-Helmholtz (biharmonic) model arising in wave dispersion of atomic models \cite{Lotfinia02012021}, nonlocal strain-gradient elasticity \cite{LAZAR20061404} and stream function formulations of inertial waves \cite{Damien22}. We consider Cauchy boundary \cite{Lotfinia02012021} and the unknown parameter $k=k(x)\in X$ appearing nonlinearly: 
\begin{align*}
\begin{cases}
\Delta^2u - k^2\Delta u + u= f \quad \text{in } \dom\\
\partial_n u = u=0 \quad \text{on }\partial \dom
\end{cases}
\end{align*}
with exterior normal vector $n$. The passive imaging problem for $k$ has the backpropagator
\begin{align*}
F'(k)^*y =  4I_X\Re \(k \int_\dom\Delta_x\(\overline{\cov(u,u)}\)\Psi\,dx'\) \qquad \forall y\in\Ycov,
\end{align*}
where the extended adjoint state $\Psi$ solves
\begin{align*}
\forall^\text{a.e.}x'\in\dom:
\begin{cases}
\Delta^2_x\Psi(\cdot,x') - \Delta_x(k^2\ \Psi(\cdot,x')) + \Psi(\cdot,x')= y(\cdot,x') \quad \text{in } \dom\\
\partial_n \Psi(\cdot,x') =\Psi(\cdot,x') = 0 \quad \text{on }\partial \dom.
\end{cases}
\end{align*}
In practice, one often finds $\theta:=k^2$ for simplicity. 
\end{example}

\section{Nonlinearity analysis}\label{sec:nonlinearity}

We now turn our attention to a different topic, namely the degree of nonlinearity in the problem. This is particular relevant for passive imaging since, as discussed in Section \ref{sec:intro}, not only the parameter-to-state map, but also the covariance observation map is nonlinear. If the model is too nonlinear, there is in general no guarantee for convergence of reconstruction. Thus, there is a need to quantify the nonlinearity inherent in the model; moreover, this quantification should be suitable for ill-posed problems, and account for the PDE aspect. 

Motivated by the tangential cone condition invented in \cite{scherzer95}, we will make two additional assumptions (Assumptions \ref{ass:bound-Dinv} and \ref{ass:space-lift}) on the regularity of the forward model, yielding a TCC-like bound on the linearization error.

\subsection*{The tangential cone condition}
The tangential cone condition (TCC) for a general model $F:X\to \Yc$ requires that 
\begin{equation}\label{TCC-general}
\begin{split}
\|F(\theta)-F(\thedag)-F'(\theta)(\theta-\thedag)\|_\Yc \leq& c_{tc}\|F(\theta)-F(\thedag)\|_\Yc \quad \forall \theta\in\Bc^X_\rho(\thedag), c_{tc}<1
\end{split}
\end{equation}
holds locally around a ground truth $\thedag\in X$ and that the TCC constant $c_{tc}$ must be strictly less than one. Unlike the Taylor expansion, the expression \eqref{TCC-general} constrains the first order linearization error relative to the image difference. This measure particularly suits nonlinear inverse problems, whose ill-posedness is caused by compactness, which is the case here by Corollary \ref{cor:compactness}; see also Appendix \ref{appendix}. Indeed, when $F$ is compact, despite large difference in the arguments $\|\theta-\thedag\|$, the image difference $\|F(\theta)-F(\thedag)\|$ can be rather small, or even zero; thus, an error bound such as that offered by the Taylor expansion is not particular helpful.

From  \eqref{TCC-general}, it follows immediately by the triangle inequality that
\[
\frac{1}{1+c_{tc}}\|F'(\theta)(\thedag-\theta)\| \leq \|F(\thedag)-F(\theta)\| \leq \frac{1}{1-c_{tc}}\|F'(\theta)(\thedag-\theta)\|;
\]
in other words, the graph of $F$ must lie within the double cone whose slopes are decided by the tangential cone constant (Figure \ref{bg-fig-tcc}). This nonlinearity condition guarantees convergence for a wide range of gradient-based iterative regularization methods \cite[Theorem 2.6]{KalNeuSch08}, and even uniqueness of a minimal-norm solution \cite[Proposition 2.1]{KalNeuSch08}. 

\begin{figure}[H] 
\centering
\includegraphics[width=9cm]{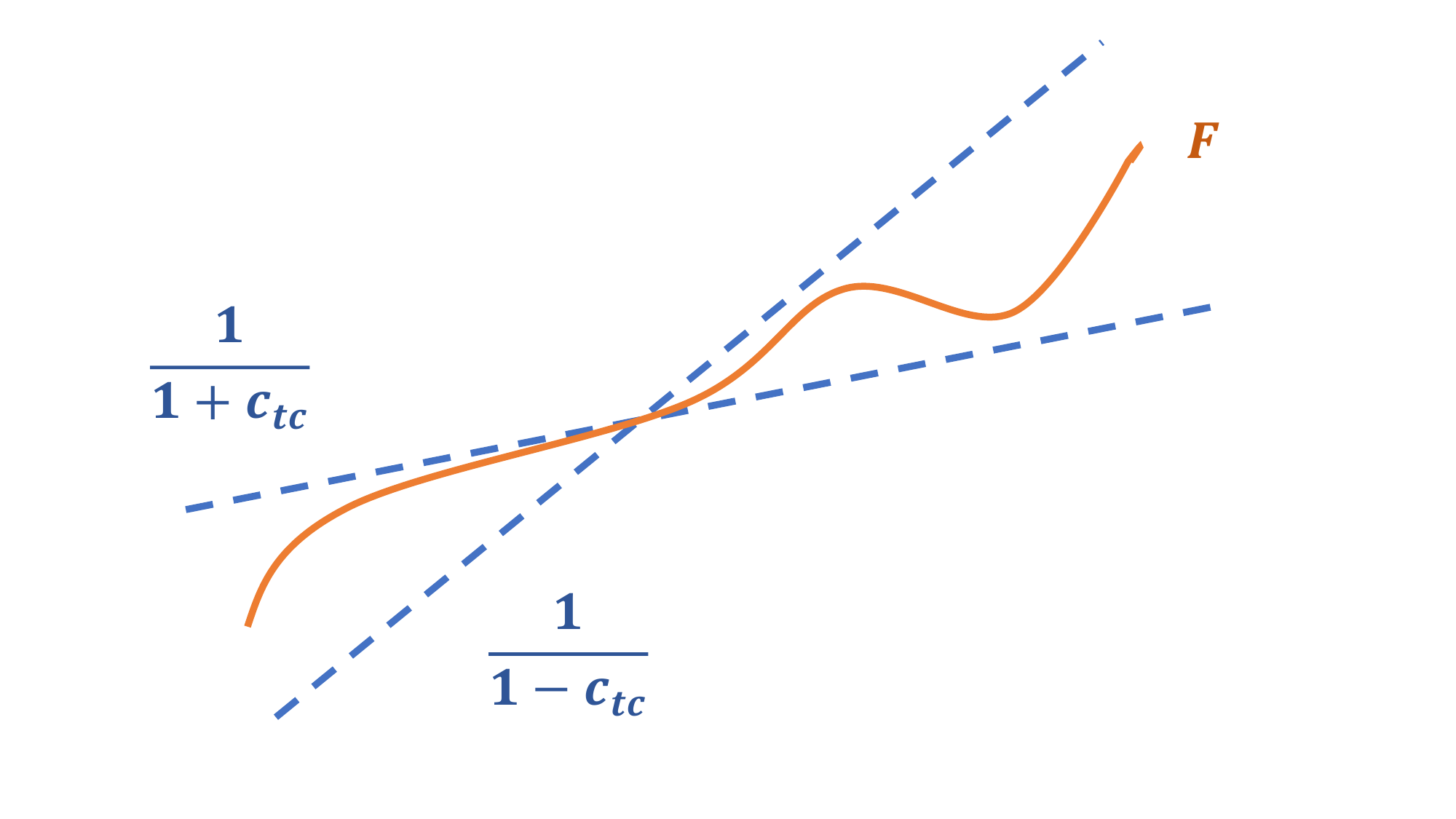}
\caption{Tangential cone condition (TCC) }
\label{bg-fig-tcc}

\end{figure} 
Despite its potency, this condition is highly technical to verify; we refer to \cite{TCC21} for a general verification strategy. Interesting research concerning the TCC includes elliptic inverse problems \cite{HoffmanWaldNguyen:2021}, quantitative elastography \cite{Nakamura:MRE21, HubmerScherzer:TCC18}, electrical impedance tomography \cite{Kindermann21}, full wave form inversion \cite{EllerRolandRieder24} and recently neural networks \cite{ScherzerHofmannNashed, holler22learning_parameter_id}. All these works deal with linear measurements; recall that the parameter-to-state map $S$ is always nonlinear.
 
\subsection*{TCC-like linearization error}

Inspired by the TCC, we shall establish a linearization error bound for the covariance observation, keeping in mind that it squares nonlinearity of the parameter-to-state map. Explicitly, we aim to show an inequality on the form
\begin{equation}\label{Elin-general}
\begin{split}
E_{\text{lin}}:=\|F(\theta)-F(\thetil)-F'(\theta)(\theta-\thetil)\| \leq C\|\theta-\thetil\|\|F(\theta)-F(\thetil)\| + N\|\theta-\thetil\|^2 %\text{for all } \theta\in\Bc_\rho(\thetil)\subseteq\Dc\text{ with some } C,N> 0
\end{split}
\end{equation}
for suitable constants $C$, $N$ and with suitable constraints on the parameters $\theta$, $\thetil$. Clearly, for $\thetil:=\thedag$, a sufficiently small radius $\rho>\|\theta-\thetil\|$ and setting $c_{tc}:=C\rho< 1$, the tangential cone constant becomes strictly smaller than one, as required. This local behavior is expected for nonlinear inverse problems. Indeed, the initial guess must be sufficiently close to the ground truth $\thedag$, otherwise there is no hope for convergence of iterative reconstruction methods in general. Thus, the first term on the right hand side of \eqref{Elin-general} implies the TCC \eqref{TCC-general}. To establish the TCC fully, the second term of \eqref{Elin-general}, in the form of the Taylor expansion error, needs to vanish or be dominated again by the image $\|F(\theta)-F(\thetil)\|$. This leads to a stability estimate problem, which we leave for future work; see Remark \ref{rem:TCC}. 

The linearization error in the form of \eqref{Elin-general} clearly displays a TCC-like structure. Thus, it represents a significant step towards obtaining convergence guarantees of passive imaging via iterative regularization, a situation which has yet to be investigated.

\subsection{Auxiliary results}\label{sec:linearize-auxiliary}

Before diving into the detail, we sketch our path to the main goal, which involves the following steps:

\begin{itemize}
    \item Lipschitz properties
    \item Cross-correlation estimate
    \item Lifted regularity
    \item Linearization error
\end{itemize}

For the time being, we will turn our attention to showing the first three of these. Two sets of assumptions will be made for this purpose: Assumptions \ref{ass:bound-Dinv} and  \ref{ass:space-lift}.

We begin with the Lipschitz property of the parametric differential map $\theta\mapsto D(\theta)$ defined in \eqref{A}, and moreover of its inverse \eqref{Dinv} under suitable boundedness assumptions. 

\begin{assumption}[Bounded inverse]\label{ass:bound-Dinv}
There is a functional $\pi:\R^+\to\R^+$ that maps any bounded set to a bounded set and satisfies
\begin{align}\label{ass-bound-Dinv}
\|D(\theta)^{-1}\|_{W^*\to U}\leq \pi(\|\theta\|_X)
\end{align}
for any $\theta\in\Dc\subseteq X$.
\end{assumption}

\begin{lemma}[Lipschitz property of $D$ and its inverse] \label{lem:lipschitz}
 We have the following:
\begin{itemize}
\item[i)] The map $D:X\to \Lc(U,W^*)$ is  Lipschitz continuous, and 
\begin{align}\label{Lip-D}
\|D(\theta)-D(\thetil)\|_{U\to W^*}\leq L_D\|\theta-\thetil\|_X \quad\text{with }L_D:=\|B\|_{U\to \Lc(X,W^*)}
\end{align}
for all $\theta,\thetil\in 	X$.

\item[ii)] Let Assumption \ref{ass:bound-Dinv} hold. The inverse  $D(\cdot)^{-1}:\Dc\to\Lc(W^*,U)$ is locally Lipschitz continuous: for all $\theta\in \Bc^X_\rho(\thetil)\subseteq \Dc$, there exists a constant $L_{inv}(\|\thetil\|_X,\rho)$ satisfying
\begin{align}\label{Lip-Dinv}
\|D(\theta)^{-1}&-D(\thetil)^{-1}\|_{W^*\to U} \leq L_{inv}(\|\thetil\|_{X},\rho)\|\theta-\thetil\|_{X}.
\end{align}
Consequently, the parameter-to-state map $S$ in \eqref{S} is locally Lipschitz continuous.
\end{itemize}
\end{lemma}
\begin{proof}
i) Recall that $D(\theta)$ in \eqref{A} is affine linear in $\theta$ and is represented via the linear bounded operator $B(u)$ in  \eqref{B}. We thus have
\begin{align*}
\|D(\theta)-D(\thetil)\|_{U\to W^*}&%=\sup_{\|u\|_U\leq 1}\|(D(\theta)-D(\thetil))u\|_{W^*}
=\sup_{\|u\|_U\leq 1}\|B(u)(\theta-\thetil)\|_{W^*}%\leq \sup_{\|u\|_U\leq 1}\|B(u)\|_{\Lc(X\, W^*)}\|\theta-\thetil\|_X
\leq \sup_{\|u\|_U\leq 1}\|B\|_{U\to \Lc(X,W^*)}\|u\|_U\|\theta-\thetil\|_X \nonumber\\
&=\|B\|_{U\to \Lc(X,W^*)}\|\theta-\thetil\|_X
\end{align*}
as claimed in \eqref{Lip-D}, with the operator norm of $B$ being the Lipschitz constant.

ii) Let $D(\theta)u=f$ and $D(\thetil)v=f$. The difference $w:=u-v$ now solves the PDE
\begin{align*}
D(\theta)w=-(D(\theta)-D(\thetil))v,
\end{align*}
implying
\begin{align*}
\|D(\theta)^{-1}&-D(\thetil)^{-1}\|_{W^*\to U}=\sup_{\|f\|_{W^*}\leq 1}\|(D(\theta)^{-1}-D(\thetil)^{-1})f\|_{U}
= \sup_{\|f\|_{W^*}\leq 1}\|w\|_{U}\\
&\leq\sup_{\|f\|_{W^*}\leq 1}\|D(\theta)^{-1}\|_{W^*\to U}\|D(\theta)-D(\thetil)\|_{U\to W^*}\|D(\thetil)^{-1}\|_{W^*\to U}\|f\|_{W^*}\\
&\leq L_D\pi(\|\thetil\|_{X}+\rho)\pi(\|\thetil\|)\|\theta-\thetil\|_{X}=: L_{inv}(\|\thetil\|_{X},\rho)\|\theta-\thetil\|_{X}
\end{align*}
by employing the Lipschitz constant $L_D$ in \eqref{Lip-D} and the locally bounded inverse assumption \eqref{ass-bound-Dinv}.

As a consequence, for any $\theta,\thetil$ in the domain $\Dc$, one has
\begin{align*}
\|S(\theta)-S(\thetil)\|_U &= \|D(\theta)^{-1}f-D(\thetil)^{-1}f\|_U\leq \|D(\theta)^{-1}-D(\thetil)^{-1}\|_{W^*\to U}\|f\|_{W^*}\\
&\leq C(\|\thetil\|_{X},\rho,\|f\|_{W^*})\|\theta-\thetil\|_{X}
\end{align*}
with some constant $C$ depending on $L_{inv}(\|\thetil\|_{X},\rho)$ and the fixed source $f$.
\end{proof}

Before moving forwards, we remark that the assumption on the locally bounded inverse \eqref{ass-bound-Dinv} is very much feasible for linear PDEs.  We demonstrate this for the $a$-problem in Proposition \ref{prop:abc} with $c=b=0$, i.e.~for the equation $D(a):=-\Div(a\nabla u) = f$.
Testing this PDE with $u\in U=H^1_0(\dom)=W$ yields
\begin{align*}
&\lan f,u\ran_{W^*,W}  \geq \underline{a}\|u\|^2_U \implies \|u\|_U\leq \frac{1}{\underline{a}}\|f\|_{W^*}
\implies \|D(\theta)^{-1}\|_{W^*\to U}\leq \frac{1}{\underline{a}}=:\pi = \text{const}.
\end{align*}
This step is in fact the coercivity estimate, which is the key to unique existence of PDE solutions, via e.g.~the Lax-Milgram theorem.

We additionally highlight that the passive imaging problem is ill-posed when measurements cannot capture smoothness of the PDE solutions, a very common situation in practice. 

\begin{corollary}[Compactness and ill-posedness]\label{cor:compactness} Let Assumption \ref{ass:bound-Dinv} hold. If the state space is compactly embedded into the data space $U\compt L^2(\dom)$, then the parameter-to-covariance map $F:\Dc(\subseteq X)\to \Yc$ is compact. The inverse problem is then ill-posed.
\end{corollary}

\begin{proof}

The forward map $F=\cov\circ E_{U\to L^2}\circ\Sv$ is compact, as it is a composition of the bounded observation map $L^2(\dom)\ni u\mapsto \cov(u,u)\in L^2(\doms)$, the locally Lipschitz continuous, thus bounded, parameter-to-state map $S:\Dc\ni\theta\to u\in U$ in Lemma \ref{lem:lipschitz} together with finiteness of the vector $\Sv$, and the compact embedding $E:U\to L^2(\dom)$.

As the forward map $F$ is compact, the inverse problem is ill-posed; see Theorem \ref{theo:compact-illpose} in the Appendix.
\end{proof}

The second auxiliary result estimates the $\Yc$-norm of the cross term $\cov(v-u,u)$, which will soon appear in the main proof.

\begin{lemma}[Boundedness of cross term] \label{lem:bound-mixed-term}
Given Assumption \ref{ass:bound-Dinv}, denote $\Pi_f:=\cov(f,f)$.\\
Then there exists a constant $C(\|\thetil\|_X,\rho, \|\Pi_f\|_{\HS(W,W^*)})$ such that  
\begin{align}\label{bound-mixed-term}
\|\cov(v-u,u)\|_{L^2(\doms)}\leq C(\|\thetil\|_X,\rho, \|\Pi_f\|_{\HS(W,W^*)})\|\theta-\thetil\|_{X}
\end{align}
holds for $v:=\Sv(\thetil)$, $u:=\Sv(\theta)$ with $\theta\in\Bc^X_\rho(\thetil)\subseteq\Dc$.
\end{lemma}
\begin{proof}
In the following, we employ all the following properties: the covariance representation \eqref{F_incov}, the continuous embedding $\HS(U^*,U)\embed \HS(L^2)$ stemming from $U\embed L^2(\dom)$, the fact that the composition of a Hilbert-Schmidt operator $A\in \HS(Y,Z)$ and a bounded operator $B\in\Lc(X,Y)$ is again Hilbert-Schmidt \cite{MeiseVogt92} with $\|AB\|_{\HS(X,Z)}\leq\|A\|_{\HS(Y,Z)}\|B\|_{X\to Y}$, the duality equivalence $\|(D(\theta)^{-1})^\star\|_{U^*\to W}=\|D(\theta)^{-1}\|_{W^*\to U}$,  
and the Lipschitz properties derived in \eqref{Lip-D}-\eqref{Lip-Dinv}. Consequently,
\begin{align*}
&\|\cov(v-u,u)\|_{L^2(\doms)}= \|\cov(v-u,u)\|_{\HS(L^2)}\\
&\quad= C_{\HS(U^*,U)\to\HS(L^2)}\|(D(\thetil)^{-1}-D(\theta)^{-1})\cov(f,f)(D(\theta)^{-1})^\star\|_{HS(U^*,U)}\\
&\quad\leq C_{\HS(U^*,U)\to\HS(L^2)}\|D(\thetil)^{-1}-D(\theta)^{-1})\|_{W^*\to U}\|\cov(f,f)\|_{\HS(W,W^*)}\|(D(\theta)^{-1})^\star\|_{U^*\to W}\\
&\quad\leq C_{\HS(U^*,U)\to\HS(L^2)}\|\cov(f,f)\|_{\HS(W,W^*)}\pi(\|\thetil\|_X+\rho) L_{inv}(\|\thetil\|_{X},\rho)\|\theta-\thetil\|_{X}\\
&\quad=: C(\|\thetil\|_X,\rho, \|\Pi_f\|_{\HS(W,W^*)})\|\theta-\thetil\|_{X}
\end{align*}
for any given fixed $\cov(f,f)\in \HS(W,W^*)$.
\end{proof}

As alluded to in the beginning of this section, the third auxiliary step involves lifting the PDE residual space from $(L^2(\dom)\subseteq) W^*$ to $L^2(\dom)$; this leads to smoother PDE solutions, provided that the parameters are smoother.

\begin{assumption}[Lifted regularity]\label{ass:space-lift}\,
\begin{enumerate}[label=\roman*)]
\item \label{ass-space-lift-1} 
Consider Hilbert spaces $\Uhat(\subseteq U)$ and $\Xhat(\embed X)$ with the Gelfand triple property and let
\begin{equation}\label{ass-space-lift}
\begin{split}
D(\theta)\in \Lc(\Uhat,L^2(\dom)),%\quad B(u)\in\Lc(\Xhat,L^2(\dom)), 
\quad B\in\Lc(\Uhat,\Lc(\Xhat,L^2(\dom))) \quad \text{for all }\theta\in\Xhat\\
D(\theta)^{-1}\in \Lc(L^2(\dom),\Uhat)\quad \text{for all }\theta\in \Dc\cap \Xhat=:\widehat{\Dc}.
\end{split}
\end{equation}

\item
Assume there is a functional $\pihat:\R^+\to\R^+$ that maps any bounded set to a bounded set and satisfies, for any $\theta\in \widehat{\Dc}$,
\begin{align}\label{ass-bound-Dinv-lift}
\|D(\theta)^{-1}\|_{L^2\to \Uhat}\leq \pihat(\|\theta\|_\Xhat).
\end{align}
\end{enumerate}
\end{assumption}

One example of the lifted regularity setting for the $a,b,c$-problem in Section \ref{sec:abc-problems and discussion} is
\begin{equation*}\label{abc-space-lift}
\begin{split}
&\Uhat:=H_0^2(\dom)\cap H_0^1(\dom), \quad \What^*:=L^2(\dom),\quad \Xhat_a:= H^2(\dom), \quad \Xhat_b:=H^1(\dom), \quad \Xhat_c:=H^1(\dom)
\end{split}
\end{equation*}
with some straightforward energy estimation.
Comparing to the framework \eqref{abc-space}, we clearly see the shifting/lifting regularity of the state space $U$ when the source space $\What^*$ and parameter space $\Xhat_c$ are smoother. Shifting regularity is a natural phenomenon in linear PDEs, as discussed in Discussion \ref{dis-general-space}. 

Under the new setting in Assumption \ref{ass:space-lift}, we also obtain Lipschitz properties similar to Lemma \ref{lem:lipschitz}; we use $\widehat{(\cdot)}$ to denote the altered Lipschitz constants. We choose the norm on $\Xhat$ such that $\|\cdot\|_X\leq\|\cdot\|_\Xhat$, thus naturally $\Bc_\rho^\Xhat(\thetil)\subseteq \Bc_\rho^X(\thetil)$, without the need to rescale or introduce the new radius $\hat{\rho}=\rho/C_{\Xhat\to X}$.

\begin{lemma}[Lifted Lipschitz properties] \label{lem:lipschitz-new}
We have the following Lipschitz properties on the new function spaces:
\begin{itemize}
\item[i)] Let Assumption \ref{ass:space-lift}-\ref{ass-space-lift-1} hold.
Then $D$ is Lipschitz continuous, and 
\begin{align}\label{Lip-D-lift}
\|D(\theta)-D(\thetil)\|_{\Uhat\to L^2}\leq \widehat{L}_D\|\theta-\thetil\|_\Xhat \quad\text{with }\widehat{L}_D:=\|B\|_{\Uhat\to \Lc(\Xhat,L^2)}
\end{align}
holds for all $\theta,\thetil\in \Xhat$.
\item[ii)] Let Assumption \ref{ass:space-lift} hold. For all $\theta\in \Bc^\Xhat_\rho(\thetil)\subseteq \widehat{\Dc}$, there exists a constant $\Lhat_{inv}(\|\thetil\|_\Xhat,\rho)$ satisfying
\begin{align}\label{Lip-Dinv-lift}
\|D(\theta)^{-1}&-D(\thetil)^{-1}\|_{L^2\to \Uhat} \leq \Lhat_{inv}(\|\thetil\|_{\Xhat},\rho)\|\theta-\thetil\|_{\Xhat},
\end{align}
that is, the inverse $D(\cdot)^{-1}$ is locally Lipschitz continuous near $\thetil$.
\end{itemize}
\end{lemma}

\begin{proof}
The proof is the same as Lemma \ref{lem:lipschitz}, with $L^2(\dom)$ in place of $W^*$, with $\Uhat$, $\Xhat$ in place of $U$, $X$ and with $\widehat{\Dc}$ in place of $\Dc$.
\end{proof}

\subsection{The TTC-like linearization error}\label{sec:linearize-error}

With all the auxiliary steps prepared, we are now ready for the main result.

\begin{theorem}[Linearization error]\label{theo:linearization}
Let Assumption \ref{ass:bound-Dinv} and \ref{ass:space-lift} hold. The first order linearization error $E_{\mathrm{lin}}$ of the forward map $F=\cov\circ\Sv$ at a point $\thetil\in\widehat{\Dc}$ is then bounded by
\begin{equation}\label{lin-error-final}
\begin{split}
E_{\mathrm{lin}}\leq \Mhat\|\theta-\thetil\|_\Xhat \|F(\theta)-F(\thetil)\|_\Yc + \Nhat\|\theta-\thetil\|_X\|\theta-\thetil\|_\Xhat
\end{split}
\end{equation}
for all $\theta\in \Bc^\Xhat_\rho(\thetil)\subseteq\widehat{\Dc}$, with some positive constants \[\Mhat=\Mhat(\|\thetil\|_\Xhat,\rho), \qquad \Nhat= \Nhat(\|\thetil\|_X,\rho, \|\Pi_f\|_{\HS(W,W^*)}).\]
\end{theorem}

\begin{proof}
The proof is carried out in three steps.

\textbf{Step 1 (Covariance decomposition).} Let $h:=\theta-\thetil$ and $u:=\Sv(\theta)$, $v:=\Sv(\thetil)$. The linearization error \eqref{Elin-general} for $F=\cov\circ\Sv$ now reads as
\begin{align}\label{Elin}
E_{\text{lin}}=\|\cov(u)-\cov(v)-\big[\cov(\phi,u)+\cov(u,\phi)\big]\|_{L^2(\doms)}.
\end{align} 
One recalls that the  linearization $F'(\theta)h$  can be expressed  via the linearized state $\phi=\Sv'(\theta)h$ as in \eqref{linearization} of Lemma \ref{lem:linearized-state}. Here, we simplify notation and write $\cov(u):=\cov(u,u)$, leaving the original (full) notation to highlight the cross-correlation between two different terms such as $\cov(\phi,u)$. In what follows, we shall separately study two the sub-terms of \eqref{Elin}: the covariance difference $W$ and the covariance linearization $Q$, given as
\[
W:=\cov(u)-\cov(v),\quad Q:=\cov(u,\phi)+\cov(\phi,u),
\] 
and combine them together at the end.

Firstly, for $W$, we insert the mixed term $v_j(x)\overline{u_j(x')}$ and write
\begin{align*}
W(x,x')=\sumj \big[u_j(x)-v_j(x)\big]\overline{u_j(x')} + \big[\overline{u_j(x')}-\overline{v_j(x')}\big]v_j(x)=:w_A(x,x')+\overline{w_B(x,x')}.
\end{align*}
Since $D(\theta)u_j=f_j$ and $D(\thetil)v_j=f_j$, the difference $u_j-v_j$ solves 
\begin{align*}
D(\theta)(u_j-v_j)= -\(D(\theta)-D(\thetil)\)v_j,
\end{align*}
noting that in the difference equation, the source  $f_j$ is no longer present. Multiplying two sides of this equation by $\overline{u_j(x')}$  and summing up all $j=1\ldots J$, we have that $w_A$ solves the PDE
\begin{align}\label{wa}
 D_1(\theta)w_A = -\(D_1(\theta)-D_1(\thetil)\)\cov(v,u)
\end{align}
at a.e.~$x'\in\dom$, the subscript again indicating that the differential operator $D$ acts on $w_A(\cdot,x')$. \\
Similarly, the equation for $w_B$  takes the form
\begin{align}\label{wb}
D_2(\theta)w_B = -\(D_2(\theta)-D_2(\thetil)\)\overline{\cov(v)}
\end{align}
for a.e.~$x\in\dom$. The key difference in the equations for $w_A$ and $w_B$ is their source terms; in particular, \eqref{wa} has the mixed term $\cov(v,u)$ on the right hand side. 

\textbf{Step 2 (Derivative decomposition).} Next, consider the linearization $Q$ on the form  
\[Q= \sumj  \phi_j(x) \overline{u_j(x')}+u_j(x)\overline{\phi_j(x')}=:q_A(x,x')+\overline{q_B(x,x')}.\] 
One recalls  Lemma \ref{lem:linearized-state} and  \eqref{B} for  $D(\theta)\phi_j=-B(u_j)h=-(D(\theta)-D(\thetil))u_j$. Multiplying both sides of this identity by $\overline{u_j(x')}$ and summing up all $j=1\ldots J$ yields the equation
\begin{align}\label{qa}
D_1(\theta)q_A = -\(D_1(\theta)-D_1(\thetil)\)\cov(u)
\end{align}
for $q_A$, and in the similar fashion, one gets the equation
\begin{align}\label{qb}
D_2(\theta)q_B =  -\(D_2(\theta)-D_2(\thetil)\)\overline{\cov(u)}
\end{align}
for $q_B$. With this step, we have represented the linearization in a form that does not contain the mixed covariance.

\textbf{Step 3 (Combining and estimating).} We next combine all elements of $E_{\text{lin}}$, that is, $w_A$, $w_B$, $q_A$, $q_B$ in \eqref{wa}-\eqref{qb}, in the following order:
\begin{align}\label{lin-error-decomp}
E_{\text{lin}}\leq \|\overline{q_B}-\overline{w_B}\|_{L^2(\doms)}+\|q_A-w_A\|_{L^2(\doms)},
\end{align}
noting that the complex conjugate does not affect the $L^2$-norm.

Firstly, we consider the $B$-terms. Both $w_B$ and $q_B$ in \eqref{wb}, \eqref{qb} satisfy similar PDEs,  differing only by the sources. In the following, we shall respectively employ duality, the isomorphism $I_\Uhat:\Uhat^*\to\Uhat$, the Hölder inequality and continuity of the embeddings:
\begin{align*}
& \|q_B-w_B\|_{L^2(\doms)}^2 = \|q_B-w_B\|_{L^2(L^2)}^2  
=\int_\dom \Big(\sup_{\|\varphi\|_{L^2\leq1}}\lan q_B(x,\cdot)-w_B(x,\cdot),\varphi\ran \Big)^2\,dx\\
\leq & \int_\dom \sup_{\|\varphi\|_{L^2\leq1}}\left|\lan D_2(\theta)^{-1}(D_2(\theta)-D_2(\thetil))(\overline{\cov(u)}-\overline{\cov(v)})(x,\cdot),\varphi\ran \right|^2\,dx\\
= & \int_\dom \sup_{\|\varphi\|_{L^2\leq1}}\left|\lan (\overline{\cov(u)}-\overline{\cov(v)})(x,\cdot),(D_2(\theta)-D_2(\thetil))^\star (D_2(\theta)^{-1})^\star\varphi\ran \right|^2\,dx\\
\leq & \int_\dom \| (\cov(u)-\cov(v))(x,\cdot)\|^2_{ L^2}\,dx  \\
& \sup_{\|\varphi\|_{\Uhat^*\leq C_{L^2\to\Uhat^*}}}\left[C_{\Uhat\to L^2}\|I_\Uhat\|_{\Uhat^*\to \Uhat}\|D_2(\theta)-D(\thetil))^\star\|_{L^2\to \Uhat^*} \|D(\theta)^{-1})^\star\|_{\Uhat^*\to L^2}\|\varphi\|_{\Uhat^*}\right]^2\\
\leq & \left[C_{L^2\to\Uhat^*}C_{\Uhat\to L^2}\|I_\Uhat\| \|D(\theta)^{-1}\|_{L^2\to \Uhat}\|D(\theta)-D(\thetil)\|_{\Uhat\to L^2}\|\cov(u)-\cov(v)\|_{L^2(\doms)}\right]^2.
\end{align*}
The above computations highlight the role of lifted regularity in Assumption \ref{ass:space-lift}, as one needs to arrive at the $L^2$-norm, in accordance with the data space, on both sides of the estimate \eqref{lin-error-1}.

With  Lipschitz continuity of $D$ in \eqref{Lip-D-lift} and the bounded inverse assumption \eqref{ass-bound-Dinv-lift}, one achieves
\begin{align}\label{lin-error-1}
\|q_B-w_B\|_{L^2(\doms)}&\leq  C_{L^2\to\Uhat^*}C_{\Uhat\to L^2}\|I_\Uhat\| \pihat(\|\thetil\|_\Xhat+\rho)\widehat{L}_D\|\theta-\thetil\|_\Xhat  \|\cov(u)-\cov(v)\|_{L^2(\doms)} \nonumber\\
&=: \Mhat(\|\thetil\|_\Xhat,\rho)\|\theta-\thetil\|_\Xhat \|\cov(u)-\cov(v)\|_{L^2(\doms)}.
\end{align}

Now, for the $A$-terms $w_A$ in \eqref{wa} and $q_A$ in \eqref{qa}, we proceed in the same fashion as the $B$-terms in the first estimate below. However, since the source term of $w_A$ contains the mixed covariance $\cov(v,u)$, we need further evaluation, in particular, the bound \eqref{bound-mixed-term} derived in Lemma \ref{lem:bound-mixed-term} in the second estimate below. More precisely, one obtains
\begin{align}\label{lin-error-2}
\|q_A-w_A\|_{L^2(\doms)}&\leq \Mhat(\|\thetil\|_\Xhat,\rho)\|\theta-\thetil\|_\Xhat \|\cov(u,u)-\cov(v,u)\|_{L^2(\doms)}\hspace{3cm} \nonumber\\
&\leq \Mhat(\|\thetil\|_\Xhat,\rho)\|\theta-\thetil\|_\Xhat  C(\|\thetil\|_X,\rho, \|\Pi_f\|_{\HS(W,W^*)})\|\theta-\thetil\|_{X} \nonumber\\
&=: \Nhat(\|\thetil\|_\Xhat,\rho, \|\Pi_f\|_{\HS(W,W^*)}) \|\theta-\thetil\|_X\|\theta-\thetil\|_\Xhat
\end{align}
with a quadratic rule for $\|\theta-\thetil\|$. Note that one can pass through the embedding $\Xhat\embed X$ one more time to get the bound in $\|\cdot\|_\Xhat$; however, we retain the sharper estimates in $X$- and $\Xhat$-norms. 

Finally, composing \eqref{lin-error-1} and \eqref{lin-error-2} into \eqref{lin-error-decomp}, we obtain the linearization error $E_{\text{lin}}$ as claimed in \eqref{lin-error-final}. The proof is complete.
\end{proof}

We close the section by discussing the link to the tangential cone condition.

\begin{discussion}\label{rem:TCC}\,
\begin{itemize}
\item 
Without the lifted regularity setting, one must deal with a smooth data space, which is not as realistic as $L^2$. Thus, a regularity-lifting strategy is typically a prerequisite for proving the TCC.

\item Theorem \ref{theo:linearization} forms an evident connection to the TCC. Further work is required to bridge this gap, including injectivity and closed range properties of the linearization \cite{Kindermann17},  and conditional stability estimates \cite{nguyen24}. To this end, the error bound by the weak norm $\|\cdot\|_X$ is likely to prove highly valuable. 
 
\item It is noteworthy that if the cross term $\cov(u-v,u)$ vanishes, the TCC immediately follows. This leads to an interesting question: which source models can imply this property? 

\end{itemize}
\end{discussion}

\section{Conclusion and outlooks}\label{sec:conclusion}
This article provides a general framework for passive imaging using covariance data. In contrast to traditional imaging tasks, covariance measurement squares dimensionality and nonlinearity in the models, making the inverse problems highly challenging. We have developed a backpropagation strategy via the extended adjoint state that is suitable, systematic and transparent to compute for any linear PDE. In addition, we have analyzed the extreme nonlinearity of the passive imaging problem, revealing a clear connection to the tangential cone condition, paving the way for obtaining convergence guarantees in iterative reconstruction. Going forwards, we plan to explore the following directions:

\begin{itemize}

\item The next step is to implement the proposed adjoint state-based approach for real-world applications. Preliminary promising result has been attained for inferring solar differential rotation and viscosity using correlation of inertial waves \cite{NguyenHohageFournierGizon, nguyeninertial}.

\item 
It is our ambition to fully establish the tangential cone condition and provide a general verification strategy as in \cite{TCC21}. This would be of fundamental importance for quantitative passive imaging using iterative solvers.

\item Simultaneously, the obtained linearization error suggests to investigate inversion convergence guarantees under this \enquote{relaxed} tangential cone condition. This appears to be a highly interesting and open question for regularization theory.

\end{itemize}

Finally, covariance matrices in fact live on manifolds that can be endowed with complex geometric structure, such as log-Euclidean, Riemannian and affine invariant metrics or the scaling-rotation distance \cite{Dryden09, Jung15}. Integrating these elements into our future work would elevate our perspective to encompass a much wider and more interdisciplinary viewpoint.

%% file: appendix.tex
\section*{Appendix}\label{appendix}

\begin{definition}[Ill-posedness]\label{def-illposed}
A model equation $F(x)=y$ is called \enquote{well-posed} in the sense of Hadamard if it satisfies three  properties:
\begin{enumerate}[label=(\arabic*)]
\item \ul{Existence:} solutions exist.
\item \ul{Uniqueness:} any solution is unique.
\item \ul{Stability:} the solution depends continuously on the data.
\end{enumerate}
If any of these conditions are violated, we call the model equation \enquote{ill-posed}.
\end{definition}
%Most of the inverse problem are ill-posed due to its nature. 

\begin{theorem}[Compact operators are always ill-posed]\label{theo:compact-illpose}
Let $F:X\to Y$ be a linear compact operator between normed spaces, and let the quotient space $X/ker(F)$ be infinite dimensional.
Then, there exists a sequence $\{x_n\}\subset X$ such that the image converges with $F(x_n)\to 0$, while $\|x_n\|\to\infty$. In particular, if $F$ is injective, then the inverse $F^{-1}:F(X)(\subseteq Y) \to X$ is unbounded.
\end{theorem}
\begin{proof}
\cite[Theorem 1.17]{Kirsch}.
\end{proof}
This means that for compact forward operators $F$, even a small amount of noise in the measurement can lead to arbitrarily large reconstruction error. This phenomenon is known as \emph{instability}, i.e.~the converse of the third property in Definition \ref{def-illposed}.